\documentclass[final,onefignum,onetabnum]{siamart251216}

\usepackage{amsmath,amsfonts,amsopn,amssymb}
\usepackage{graphicx}
\usepackage[caption=false]{subfig}
\graphicspath{ {./Figures/} }
\usepackage{multirow,relsize,makecell}
\usepackage{url}
\ifpdf
\DeclareGraphicsExtensions{.eps,.pdf,.png,.jpg}
\else
\DeclareGraphicsExtensions{.eps}
\fi
\usepackage{algorithm, algorithmic}
\usepackage[shortlabels]{enumitem}
\usepackage{accents}
\usepackage{diagbox}

\renewcommand{\bf}{\boldsymbol{f}}

\newcommand{\bu}{\boldsymbol{u}}
\newcommand{\bv}{\boldsymbol{v}}
\newcommand{\bw}{\boldsymbol{w}}

\newcommand{\bI}{\boldsymbol{I}}
\newcommand{\bK}{\boldsymbol{K}}

\newcommand{\cN}{\mathcal{N}}
\newcommand{\cT}{\mathcal{T}}
\newcommand{\cV}{\mathcal{V}}

\renewcommand{\div}{\operatorname{div}}

\numberwithin{theorem}{section}
\newsiamremark{assumption}{Assumption}
\newsiamremark{remark}{Remark}
\newsiamremark{question}{Question}
\newsiamremark{example}{Example}
\crefname{assumption}{Assumption}{Assumptions}
\crefname{remark}{Remark}{Remarks}
\crefname{example}{Example}{Examples}

\usepackage[normalem]{ulem}
\usepackage{color}

\headers{Multilevel methods for Forchheimer}{Jongho Park and S.~Majid Hassanizadeh}

\title{Parallel multilevel methods\\ for solving the Darcy--Forchheimer model\\ based on a nearly semicoercive formulation\thanks{This work was conducted while the second author was a visiting professor at KAUST, hosted by the Computer, Electrical and Mathematical Sciences and Engineering Division.}}

\author{
Jongho Park\thanks{Applied Mathematics and Computational Sciences Program, Computer, Electrical and Mathematical Sciences and Engineering Division, King Abdullah University of Science and Technology~(KAUST), Thuwal 23955, Saudi Arabia
(\email{jongho.park@kaust.edu.sa}).}
\and
S.\ Majid Hassanizadeh\thanks{Department of Earth Sciences, Utrecht University, 3584, CB, Utrecht, the Netherlands 
(\email{S.M.Hassanizadeh@uu.nl}).}
}

\ifpdf
\hypersetup{
  pdftitle={Parallel multilevel methods for solving the Darcy--Forchheimer model based on a nearly semicoercive formulation},
  pdfauthor={Jongho Park and S. Majid Hassanizadeh}
}
\fi

\begin{document}


\maketitle
\begin{abstract}
High-velocity fluid flow through porous media is modeled by prescribing a nonlinear relationship between the flow rate and the pressure gradient, called the Darcy--Forchheimer equation.
This paper is concerned with the analysis of parallel multilevel methods for solving the Darcy--Forchheimer model.
We begin by reformulating the Darcy--Forchheimer model as a nearly semicoercive convex optimization problem via the augmented Lagrangian method.
Building on this formulation, we develop a parallel multilevel method, also known as a multilevel additive Schwarz method, within the framework of subspace correction for nearly semicoercive convex problems, yielding a theoretically supported and computationally efficient solver for the Darcy--Forchheimer model.
The convergence analysis establishes robustness with respect to the augmented Lagrangian parameter $\epsilon$.
To further enhance convergence, we incorporate a backtracking line search and a full approximation scheme.
Numerical results support the theoretical findings and demonstrate the effectiveness of the proposed approach.
\end{abstract}



\section{Introduction}
\label{Sec:Introduction}
The Darcy--Forchheimer model characterizes a nonlinear relationship between flow rate and pressure gradient in porous media, and arises in a variety of important applications~(see, e.g.,~\cite{SMH:2023}).
A common situation relates to groundwater flow in high-permeability zones, such as karstic aquifers~(see~\cite{MW:2021}).
Also, the flow velocity is high around wellbores in oil and gas reservoirs, in geothermal reservoirs, or in CO$_2$ sequestration systems, where fluids are extracted or injected~(cf.~\cite{CDQHLQXL:2023}).
Another example is in gas reservoirs where high compressibility and velocity make inertial effects important~\cite{MDG:2015,YLQ:2015}.
In some groundwater contamination systems, such as reactive barriers and funnel and gate systems, high groundwater velocity is encountered~\cite{GGSJH:2000}.
Finally, in some manufactured porous media, such as porous catalysts or packed columns, high-speed gas or liquid flows lead to non-Darcy behavior~\cite{HRD:2012,ZZ:1988}.
Various studies have shown that standard groundwater models based on the Darcy model fail to predict flow fields in the systems described above.
Often, they underestimate velocity by orders of magnitude for a given pressure distribution; see, e.g.,~\cite{Kuniansky:2016}.

Due to its significance, substantial research has been dedicated to developing numerical methods for the Darcy--Forchheimer model.
Regarding numerical discretizations, many studies have focused on finite element approaches, which are often classified based on the choice of weak formulation.
In~\cite{PR:2012,Park:2005}, mixed finite element methods were developed using a weak formulation with the primal space $H(\div; \Omega) \cap L^3(\Omega)^d$ and the dual space $L^2(\Omega)$, where $d$ denotes the spatial dimension of the domain $\Omega$.
In contrast, studies such as~\cite{GW:2008,LMS:2009,SLM:2013} adopted a weak formulation with the primal space $L^3(\Omega)^d$ and the dual space $W^{1, \frac{3}{2}}(\Omega)$.
For the corresponding discretization, several numerical solvers have been proposed.
A Peaceman--Rachford-type iterative method was proposed in~\cite{GR:2008}, while a nonlinear multigrid method for the coupled nonlinear system was developed in~\cite{HCR:2018}.
Furthermore, a monolithic multigrid method for coupled problems on fractured domains was introduced in~\cite{AGPR:2019}.
In~\cite{FA:2020}, a two-level discretization based on a mixed formulation was investigated to reduce the computational cost associated with the nonlinear component.
More recently, transformed primal-dual methods with variable preconditioners were studied in~\cite{CGW:2023}.

The aim of this paper is to develop efficient numerical methods for the Darcy--Forchheimer model.
In particular, we adopt a multilevel approach, building on the well-established philosophy of leveraging hierarchical solvers for nonlinear systems.
Multilevel methods are among the most powerful tools for designing numerical solvers in scientific computing and have been extensively studied.
Representative works include Bramble--Pasciak--Xu preconditioners~\cite{BPX:1990} and multilevel Schwarz methods~\cite{DSW:1996,Zhang:1992,Zhang:1994}, in which rigorous convergence theories for multilevel methods were developed.
Multilevel methods have also been successfully applied to linear problems involving vector fields; see, e.g.,~\cite{AFW:2000,HT:2000}.
In addition to the aforementioned studies~\cite{AGPR:2019,HCR:2018}, multilevel techniques have been successfully applied to a wide range of nonlinear problems.
The design and analysis of multilevel methods for nonlinear problems remain active areas of research~(see, e.g.,~\cite{BF:2024}).

We note that a number of multilevel methods have been proposed for the Darcy--Forchheimer model, or more generally for nonlinear saddle point problems. 
Although these methods demonstrate excellent numerical performance in practice, theoretical guarantees for fast global convergence remain limited.
For example, Newton--Krylov--Schwarz~\cite{CGKMY:1998,CGKT:1994} and nonlinear Schwarz preconditioning~\cite{CK:2002,CKM:2002,DGKKM:2016}, originally designed for general nonlinear problems, also perform well for the Darcy--Forchheimer model. 
However, since the nonlinearity is treated by Newton iterations, only fast local convergence can typically be guaranteed; see, e.g.,~\cite{BV:2004}.

Some multigrid methods tailored to the Darcy--Forchheimer model were considered in~\cite{AGPR:2019,HCR:2018}. 
These works adopt a monolithic multigrid strategy that applies the multilevel approach directly to the nonlinear saddle point formulation, making convergence analysis difficult due to the complexity of the problem.
We also note related multigrid methods for other nonlinear saddle point problems, such as the Cahn--Hilliard-type equations~\cite{AKW:2013,Wise:2010}, where convergence analysis is also unavailable.

Motivated by the above observations, we propose multilevel methods for the Darcy--Forchheimer model whose numerical performance is supported by mathematical analysis.
The main idea is to reformulate the Darcy--Forchheimer model as a convex optimization problem using the augmented Lagrangian method~\cite{Hestenes:1969,Powell:1969}.
While this strategy was originally proposed for linear problems~\cite{LWC:2009,LWXZ:2007}, we extend it here to the nonlinear setting.
Once the model is cast into a convex optimization framework, we design parallel multilevel solvers grounded in the theory of subspace correction methods for convex optimization~\cite{Park:2020,Park:2022a,TX:2002}, which have been proven effective for a broad class of nonlinear partial differential equations~\cite{LP:2025a,Park:2024a}.
In particular, the resulting convex optimization problem is nearly semicoercive, i.e., it becomes ill-conditioned as the augmented Lagrangian parameter $\epsilon$ tends to zero.
To address this challenge, we adopt the recently proposed framework of subspace correction methods tailored to nearly semicoercive convex optimization problems~\cite{LP:2025b}.
In summary, although the individual ingredients---such as the augmented Lagrangian method, nearly semicoercive convex optimization, and the multilevel subspace correction---are well known, the novelty lies in combining these tools into an analyzable and practically efficient solver for the Darcy--Forchheimer model.

The main contributions are as follows:
\begin{itemize}
\item We show that the coupled nonlinear Darcy--Forchheimer model can be reformulated as a nearly semicoercive convex optimization problem via the augmented Lagrangian approach.
\item Based on this formulation, we design theoretically supported and computationally efficient parallel multilevel methods for the Darcy--Forchheimer model by combining nearly semicoercive subspace correction theory with a multilevel subspace decomposition.
\end{itemize}

The remainder of this paper is organized as follows.
In \cref{Sec:Model}, we review finite element discretizations for the Darcy--Forchheimer model.
In \cref{Sec:Reduction}, we present the reformulation of the Darcy--Forchheimer model as a nearly semicoercive convex optimization problem.
In \cref{Sec:Multilevel}, we propose and analyze relevant parallel multilevel methods.
In \cref{Sec:Numerical}, we report numerical results that demonstrate the performance of the proposed methods.
Finally, in \cref{Sec:Conclusion}, we provide a summary and conclusions.

\section{The Darcy--Forchheimer model}
\label{Sec:Model}
In this section, we briefly review the mathematical formulation of the Darcy--Forchheimer model and its finite element discretization.
Relevant references include~\cite{GW:2008,HCR:2018,PR:2012}.

\subsection{The continuous problem}
Let $ \Omega \subset \mathbb{R}^d $ be a bounded polyhedral domain with $ d = 2,3 $.
The steady-state Darcy--Forchheimer model is governed by the following equations.
The nonlinear relationship between the Darcy velocity $\bu$ and pressure $p$~\cite{HG:1987} is given by
\begin{subequations}
\label{Forchheimer}
\begin{equation}
\label{Forchheimer_relationship}
\frac{\mu}{\rho} \bK^{-1} \bu + \frac{\beta}{\rho} | \bu | \bu + \nabla p = \bf \quad \text{in } \Omega,
\end{equation}
where $\bK$ is the permeability tensor, $\rho$ is the fluid density, $\mu$ is the fluid viscosity, $\beta$ is the Forchheimer coefficient, and $\bf$ is the external body force per unit volume.

The conservation of mass is expressed as
\begin{equation}
\div \bu = g \quad \text{ in } \Omega,
\end{equation}
where $g$ is a prescribed source or sink term representing mass production or depletion.


For clarity and simplicity, we impose the pure homogeneous Dirichlet boundary condition:
\begin{equation}
\label{Forchheimer_BC}
p = 0 \quad \text{on } \partial \Omega.
\end{equation}
\end{subequations}

\subsection{Finite element approximations}
Various finite element approximations have been proposed in the literature for approximating the Darcy--Forchheimer model~\cite{GW:2008,LMS:2009,PR:2012,Park:2005,SLM:2013}; see also~\cite{HCR:2018} and the references therein.  
Among these, we adopt the approach introduced in~\cite{PR:2012}, which employs Raviart--Thomas-type finite element discretizations~\cite{BBF:2013}.  

We define the function spaces  
\begin{align*}
X &= \left\{ \bv \in L^3(\Omega)^d : \div \bv \in L^2(\Omega) \right\}, \\
M &= L^2(\Omega).
\end{align*}
Observe that $X$ is the intersection of two Banach spaces, $H (\div; \Omega)$ and $L^3(\Omega)^d$:
\begin{equation*}
X = H (\div; \Omega) \cap L^3(\Omega)^d.
\end{equation*}
Equipped with the norm
\begin{equation*}
\| \bv \|_X = \left( \| \div \bv \|_{L^2(\Omega)}^2 + \| \bv \|_{L^2(\Omega)}^2 + \| \bv \|_{L^3(\Omega)}^2 \right)^{1/2},
\end{equation*}
the space $X$ becomes a uniformly smooth and uniformly convex Banach space (cf.~\cite{Megginson:1998}).  

The weak formulation of~\eqref{Forchheimer} defined on $(X, M)$ is as follows: find $ ( \bu, p) \in X \times M $ such that  
\begin{equation}
\label{weak}
\begin{split}
\frac{\mu}{\rho} \int_{\Omega} \bK^{-1} \bu \cdot \bv \,dx + \frac{\beta}{\rho} \int_{\Omega} | \bu | \bu \cdot \bv \,dx - \int_{\Omega} p \div \bv \,dx = \int_{\Omega} \bf \cdot \bv \,dx
&\quad \forall \bv \in X, \\
\int_{\Omega} q \div \bu \,dx = \int_{\Omega} gq \,dx 
&\quad \forall q \in M.
\end{split}
\end{equation}


Let $\cT_h$ be a quasi-uniform partition of $\Omega$ into finite elements, where $h$ denotes the characteristic mesh size.
Let $X_h \times M_h \subset X \times M$ be a Raviart--Thomas-type conforming mixed finite element space~\cite{BBF:2013} defined on $\cT_h$.
For example, in two dimensions ($d = 2$), if each element $T \in \cT_h$ is a rectangle, the local spaces $X_h (T)$ and $M_h (T)$ are defined as
\begin{equation}
\label{RT_mixed}
X_h (T) := Q_{k+1, k} (T) \oplus Q_{k, k+1} (T),
\quad
M_h (T) := Q_{k, k} (T),
\end{equation}
for $k \geq 0$, where \(Q_{k,\ell}(T)\) denotes the space of polynomials on \(T\) of degree at most \(k\) in the first coordinate and at most \(\ell\) in the second coordinate.
Additional examples of Raviart--Thomas-type conforming mixed finite element spaces can be found in~\cite[Table~1]{PR:2012}.

The finite element formulation of~\eqref{weak} using the discrete space $X_h \times M_h$ is given as follows: find $(\bu_h, p_h) \in X_h \times M_h$ such that
\begin{equation}
\label{mixed_FEM}
\begin{split}
\frac{\mu}{\rho} \int_{\Omega} \bK^{-1} \bu_h \cdot \bv \,dx 
+ \frac{\beta}{\rho} \int_{\Omega} | \bu_h | \bu_h \cdot \bv \,dx 
- \int_{\Omega} p_h \div \bv \,dx 
= \int_{\Omega} \bf \cdot \bv \,dx \quad &\forall \bv \in X_h, \\
\int_{\Omega} q \div \bu_h \,dx 
= \int_{\Omega} g q \,dx \quad &\forall q \in M_h.
\end{split}
\end{equation}
Existence and uniqueness results for~\eqref{mixed_FEM} can be found in~\cite[Theorem~3.5]{PR:2012}.  

We define the functional $F \colon X \rightarrow \mathbb{R}$ by
\begin{equation}
\label{F}
F(\bv) = \frac{\mu}{2 \rho} \int_{\Omega} \bK^{-1} \bv \cdot \bv \,dx + \frac{\beta}{3\rho} \int_{\Omega} | \bv |^3 \,dx - \int_{\Omega} \bf \cdot \bv \,dx,
\quad \bv \in X.
\end{equation}
One readily observes that~\eqref{mixed_FEM} corresponds to the optimality condition of the following saddle-point problem:
\begin{equation}
\label{mixed_FEM_saddle}
\min_{\bv \in X_h} \max_{q \in M_h}
\left\{ F(\bv) - \int_{\Omega} q (\div \bv - g ) \,dx \right\}.
\end{equation}
By invoking Fenchel--Rockafellar duality~\cite{JPX:2025,Rockafellar:1970}, the saddle-point problem~\eqref{mixed_FEM_saddle} is equivalent to a constrained convex optimization problem.  
Namely,~\eqref{mixed_FEM} can be reformulated as a constrained minimization problem, as summarized in \cref{Prop:equiv_opt}.

\begin{proposition}
\label{Prop:equiv_opt}
The function $\bu_h \in X_h$ is a primal solution of~\eqref{mixed_FEM} if and only if it solves the following convex optimization problem with a linear constraint:
\begin{equation}
\label{mixed_FEM_opt}
\min_{\bv \in X_h} F(\bv) \quad
\text{ subject to } \quad
\div \bv = g_h,
\end{equation}
where the functional $F$ is defined in~\eqref{F}, and $g_h$ is the $L^2 (\Omega)$-orthogonal projection of $g$ onto $M_h$.
\end{proposition}

\section{Reduction to a nearly semicoercive convex optimization problem}
\label{Sec:Reduction}
In this section, we show that the finite element discretization~\eqref{mixed_FEM} can be reduced to a nearly semicoercive convex optimization problem if we proceed as in~\cite[Section~2.2]{LWXZ:2007}.
More precisely, we show that an iteration of the augmented Lagrangian method~\cite{Hestenes:1969,Powell:1969} for solving the constrained optimization problem~\eqref{mixed_FEM_opt} is equivalent to a nearly semicoercive convex optimization problem~\cite{LP:2025b}, where the convergence rate of the augmented Lagrangian method can be arbitrarily fast.

The augmented Lagrangian method~\cite{Hestenes:1969,Powell:1969} for solving~\eqref{mixed_FEM_opt} is summarized in \cref{Alg:aug}.

\begin{algorithm}
\caption{Augmented Lagrangian method for solving~\eqref{mixed_FEM_opt}}
\begin{algorithmic}[]
\label{Alg:aug}
\STATE Given $\epsilon > 0$:
\STATE Choose $p_h^{(0)} \in M_h$.
\FOR{$n=0,1,2,\dots$}
    \STATE $\displaystyle
    \bu_h^{(n+1)} = \operatornamewithlimits{\arg\min}_{\bv \in X_h} \left\{ F (\bv) - \int_{\Omega} p_h^{(n)} \div \bv \,dx + \frac{1}{2\epsilon} \int_{\Omega} ( \div \bv - g_h )^2 \,dx \right\}
    $
    \STATE $\displaystyle
    p_h^{(n+1)} = p_h^{(n)} - \epsilon^{-1} ( \div \bu_h^{(n+1)} - g_h)
    $
\ENDFOR
\end{algorithmic}
\end{algorithm}

In the linear case, for instance when $\beta = 0$ in~\eqref{mixed_FEM}, it was shown in~\cite[Lemma~2.1]{LWXZ:2007} that the augmented Lagrangian method can achieve arbitrarily fast convergence by choosing $\epsilon$ sufficiently small; similar results can also be found in~\cite{LP:2009,LP:2017}.
In \cref{Thm:aug}, we show that this property remains valid in the nonlinear case.
For a differentiable convex functional $G$ on a finite-dimensional space $V$, we define the \textit{symmetrized Bregman divergence}~\cite{ABB:2013,NN:2009} by
\begin{equation}
\label{symmetrized_Bregman}
D_G^{\mathrm{sym}} (v, w) = \langle G'(v) - G'(w), v - w \rangle, \quad v, w \in V,
\end{equation}
where $G'(v) \in V^*$ denotes the G\^{a}teaux derivative of $G$ at $v$, and $\langle \cdot, \cdot \rangle \colon V^* \times V \to \mathbb{R}$ denotes the duality pairing on $V$.
In \cref{Thm:aug}, we use this definition with $G=F$ and $V=X_h$.
We also note that the augmented Lagrangian method in a general abstract setting is analyzed in \cref{App:Aug}.

\begin{theorem}
\label{Thm:aug}
Let $\{ ( \bu_h^{(n)}, p_h^{(n)} ) \}$ be the sequence generated by the augmented Lagrangian method presented in \cref{Alg:aug}.
Then, there exists a constant $\mu > 0$, independent of $\epsilon$, such that
\begin{equation*}
D_F^{\mathrm{sym}} (\bu_h^{(n+1)}, \bu_h) 
\leq \frac{\epsilon}{4} \| p_h^{(n)} - p_h \|_{L^2 (\Omega)}^2
\leq \frac{\epsilon}{4} \left( \frac{\epsilon}{\mu + \epsilon} \right)^{2n} \| p_h^{(0)} - p_h \|_{L^2 (\Omega)}^2,
\quad n \geq 0,
\end{equation*}
where the symmetrized Bregman divergence $D_F^{\mathrm{sym}}$ is defined in~\eqref{symmetrized_Bregman}.
\end{theorem}
\begin{proof}
Thanks to~\eqref{aug_general_reduction}, \cref{Rem:local_strong_convexity}, and \cref{Thm:aug_general}, it suffices to verify that $F^* \circ \div^*$ is locally strongly convex on $M_h$, where $F^* \colon X_h \rightarrow \overline{\mathbb{R}}$ is the Legendre--Fenchel conjugate of $F$ defined in~\eqref{Legendre_Fenchel}, and $\div^* \colon M_h \rightarrow X_h$ denotes the $L^2(\Omega)$-adjoint of $\div \colon X_h \rightarrow M_h$.
Since $F$ is locally smooth on $X_h$~(see \cref{Lem:d1}(a)), it follows from~\cite{GR:2008} that $F^*$ is locally strongly convex.
Moreover, because $\div^*$ is injective, the composition $F^* \circ \div^*$ inherits local strong convexity.
This completes the proof.
\end{proof}

\begin{remark}
\label{Rem:epsilon}
Since \cref{Thm:aug} shows that smaller $\epsilon$ yields faster convergence of \cref{Alg:aug}, it may be desirable to choose $\epsilon$ as small as possible.
However, in practice, choosing $\epsilon$ too small may lead to numerical instability, which should be avoided for robust computation.
Based on the experiments in \cref{Sec:Numerical}, we recommend using a moderately small value of $\epsilon$; values in the range $10^{-3}$--$10^{-1}$ already yield fast augmented Lagrangian convergence, and $\epsilon=10^{-2}$ serves as a stable default choice in the tested examples.
Relevant discussions for the linear case can be found in, e.g.,~\cite{LP:2017}.
\end{remark}

\cref{Thm:aug} implies that the convergence rate of the augmented Lagrangian method for solving~\eqref{mixed_FEM_opt} can be made arbitrarily fast by choosing $\epsilon$ sufficiently small.  
In other words, the number of iterations required to achieve a prescribed level of accuracy becomes very small as $\epsilon$ decreases.
Relevant numerical results illustrating this behavior will be provided in \cref{Sec:Numerical}.

In this sense, the mixed formulation~\eqref{mixed_FEM} can be viewed as being reduced to a convex optimization problem of the form
\begin{equation}
\label{mixed_FEM_nearly}
\min_{\bv \in X_h} \left\{ F^{\epsilon} (\bv ; q_h) := \frac{1}{2} \int_{\Omega} ( \div \bv - g_h )^2 \,dx + \epsilon \left[ F (\bv) -  \int_{\Omega} q_h \div \bv \,dx \right] \right\}
\end{equation}
for some $q_h \in M_h$.
When there is no ambiguity, we simply write $F^{\epsilon} (\bv) = F^{\epsilon} (\bv; q_h)$.
In what follows, it suffices to focus on solving~\eqref{mixed_FEM_nearly}.
We denote the solution of~\eqref{mixed_FEM_nearly} by $\bu_h^{\epsilon} \in X_h$.

A major difficulty in solving~\eqref{mixed_FEM_nearly} numerically is that the problem is nearly semicoercive~\cite{LP:2025b}~(cf.\ nearly singular in the linear case~\cite{LWXZ:2007}) due to the nontrivial kernel of $\div$.  
As a result,~\eqref{mixed_FEM_nearly} becomes increasingly ill-conditioned as $\epsilon$ becomes small, causing conventional iterative solvers to converge very slowly.

In the following section, we present the construction of an efficient iterative solver for~\eqref{mixed_FEM_nearly} whose convergence is robust as \(\epsilon\) becomes small.

\section{Multilevel methods}
\label{Sec:Multilevel}
In this section, we propose a multilevel method for solving the nearly semicoercive convex optimization problem~\eqref{mixed_FEM_nearly} that is robust with respect to the augmented Lagrangian parameter~$\epsilon$.
Building upon the abstract theory of subspace correction methods for solving nearly semicoercive convex optimization problems developed in~\cite{LP:2025b}, we construct a patch-based multilevel method with Schwarz smoothers, as in~\cite{AFW:2000,LWC:2009}.
We also present a backtracking strategy for line search in the multilevel method~\cite{Park:2022a}, which makes the algorithm more practical by avoiding a predetermined step size and often improving the convergence rate.

\subsection{Multilevel subspace correction}
\begin{figure}
    \centering
    \includegraphics[width=\textwidth]{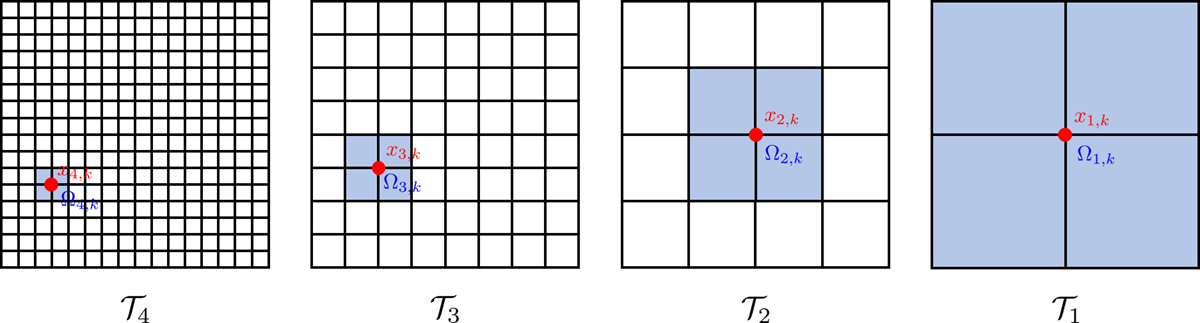}
    \caption{Multilevel mesh hierarchy $\{ \mathcal{T}_j \}_{j=1}^J$ for a rectangular grid $\mathcal{T}_h$ ($J = 4$).
    For each vertex $x_{j,k}$ of the mesh $\mathcal{T}_j$ (depicted in red), the blue region represents the corresponding subdomain $\Omega_{j,k}$.}
    \label{Fig:grids}
\end{figure}

We assume that the partition $\cT_h$ is part of a nested sequence of quasi-uniform partitions $\{ \cT_j \}_{j=1}^J$, where $\cT_h = \cT_J$ and each $\cT_j$, $1 \leq j \leq J$, has a characteristic element size $h_j$.
The quasi-uniformity constants are independent of $j$, and the mesh sizes satisfy $h_j \eqsim \gamma^j$ for some coarsening factor $\gamma \in (0,1)$.
See \cref{Fig:grids} for an example of a rectangular grid $\cT_h$ when $J = 4$.

For each $\cT_j$, we define the finite element space $V_j$, of the same type as $V := X_h$.
It follows that
\begin{equation*}
V_1 \subset V_2 \subset \dots \subset V_J = V.
\end{equation*}
Let $\cV_j$ be the set of all interior vertices of $\cT_j$, and let $n_j = | \cV_j |$.
For each $x_{j,k} \in \cV_j$, $1 \leq k \leq n_j$, we define
\begin{equation*}
\Omega_{j,k} = \bigcup \{T \in \cT_j : x_{j,k} \in \partial T \}.
\end{equation*}
We define the subspace $V_{j,k}$ as
\begin{equation*}
V_{j,k} = \{ v \in V_j : \operatorname{supp} v \subset \overline{\Omega}_{j,k} \},
\end{equation*}
so that we have
\begin{equation*}
    V_j = \sum_{k=1}^{n_j} V_{j,k}.
\end{equation*}
Then we have the following vertex-based multilevel space decomposition of $V$:
\begin{equation}
\label{space_decomposition}
V = \sum_{j=1}^J \sum_{k=1}^{n_j} V_{j,k}.
\end{equation}
Such a vertex-based space decomposition for vector field problems was previously considered in, e.g.,~\cite{AFW:2000,LWC:2009}.

Now, we are ready to present our proposed multilevel method.
The proposed method is a parallel subspace correction method~\cite{LP:2025b,Park:2020,TX:2002} for solving~\eqref{mixed_FEM_nearly} based on the multilevel space decomposition~\eqref{space_decomposition}, as summarized in \cref{Alg:multilevel}.

\begin{algorithm}
\caption{Parallel multilevel method for solving~\eqref{mixed_FEM_nearly}}
\begin{algorithmic}[]
\label{Alg:multilevel}
\STATE Given $\tau > 0$:
\STATE Choose $\bu^{(0)} \in V$.
\FOR{$n=0,1,2,\dots$}
    \FOR{$j = 1, 2, \dots, J$}
        \FOR{$k = 1, 2, \dots, n_j$}
            \STATE $\displaystyle \bw_{j,k}^{(n+1)} = \operatornamewithlimits{\arg\min}_{\bw_{j,k} \in V_{j,k}} F^{\epsilon} ( \bu^{(n)} + \bw_{j,k} )
            $
        \ENDFOR
    \ENDFOR
    \STATE $\displaystyle
    \bu^{(n+1)} = \bu^{(n)} + \tau \sum_{j=1}^J \sum_{k=1}^{n_j} \bw_{j,k}^{(n+1)}
    $
\ENDFOR
\end{algorithmic}
\end{algorithm}

Note that each local problem for $\bw_{j,k}^{(n+1)}$ in \cref{Alg:multilevel} is a convex optimization problem of small dimension, which can be solved robustly and efficiently using second-order optimization algorithms, such as the damped Newton method~\cite{BV:2004}.
Moreover, the loops over $j$ and $k$ in \cref{Alg:multilevel} can be fully parallelized, making the algorithm well-suited for implementation on parallel computing architectures.
The step size~$\tau$ in \cref{Alg:multilevel} must satisfy a certain condition (to be specified later) to ensure convergence of the algorithm.
Nevertheless, as we will show later, the backtracking line search proposed in~\cite{Park:2022a} can be employed to efficiently determine a suitable step size~$\tau$ with only marginal additional computational cost at each iteration of \cref{Alg:multilevel}.

\begin{remark}
\label{Rem:BPX}
In the linear case, i.e., when $\beta = 0$ in~\eqref{mixed_FEM_nearly}, \cref{Alg:multilevel} can serve as a preconditioner.
In fact, it coincides with the Bramble--Pasciak--Xu preconditioner~\cite{BPX:1990} or the multilevel additive Schwarz preconditioner~\cite{Zhang:1992,Zhang:1994}.
Convergence analysis of multilevel methods for linear problems posed in $H(\div)$ spaces can be found in, e.g.,~\cite{AFW:2000,HT:2000}.
\end{remark}

\begin{remark}
\label{Rem:successive}
While this paper focuses on parallel multilevel methods (multilevel additive Schwarz), one may also consider successive multilevel methods (multilevel multiplicative Schwarz), such as V-cycle multigrid.
These approaches often perform well in practice and may even achieve better convergence rates than parallel methods; see~\cite[Section~4]{CXZ:2008} and~\cite[Lemma~4.6]{XZ:2017}.
Moreover, they can be parallelized using standard coloring techniques~\cite[Section~5.1]{Park:2020}.

However, the analysis of successive subspace correction methods for nearly semicoercive problems is still lacking.
Although the linear case was studied in~\cite{LWXZ:2007}, the analysis relies on the Xu--Zikatanov identity~\cite{XZ:2002,XZ:2017}, whose extension to convex optimization remains open.
For this reason, we leave successive multilevel methods and their analysis for future work.
\end{remark}

\subsection{Convergence analysis}
We present a convergence analysis of the proposed multilevel method in \cref{Alg:multilevel}, based on the abstract theory developed in~\cite{LP:2025b}.
For completeness, a summary of the abstract framework is provided in \cref{App:Nearly}.

Throughout this section, we use the notations
$A \lesssim B$ and $B \gtrsim A$ to indicate that there exists a positive constant $C$, independent of the geometric parameters such as $h$ and $J$, such that $A \leq C B$.
We write $A \eqsim B$ when both $A \lesssim B$ and $A \gtrsim B$ hold.

First, we observe that the multilevel space decomposition~\eqref{space_decomposition} satisfies the strengthened convexity condition~\eqref{strengthened_convexity} with
\begin{equation}
\label{tau_0}
\tau_0 \gtrsim J^{-1} \eqsim | \log h |^{-1},
\end{equation}
as a consequence of a coloring argument~\cite[Section~5.1]{Park:2020} applied at each level.

Let $F_0$ and $F_1$ be the semicoercive and coercive parts of the nearly semicoercive energy functional $F^{\epsilon}$, respectively, namely,
\begin{equation}
\label{F0}
F_0 (\bv) = \frac{1}{2} \int_{\Omega} (\div \bv - g_h )^2 \,dx, \quad
F_1(\bv) = F(\bv) - \int_{\Omega} q_h \div \bv \,dx,
\quad \bv \in V,
\end{equation}
so that $F^{\epsilon} = F_0 + \epsilon F_1$.
We define the Bregman divergences associated with $ F_0 $ and $F_1$ by
\begin{equation}
\label{d1}
\begin{aligned}
d_0(\bw; \bv) &= F_0(\bv + \bw) - F_0(\bv) - \langle F_0'(\bv), \bw \rangle, \\
d_1(\bw; \bv) &= F_1(\bv + \bw) - F_1(\bv) - \langle F_1'(\bv), \bw \rangle,
\end{aligned}
\quad \bv, \bw \in V,
\end{equation}
where $\langle \cdot, \cdot \rangle$ denotes the duality pairing on $V$.
By a direct computation, we observe that
\begin{equation}
\label{d0}
d_0(\bw; \bv) = \frac{1}{2} \| \div \bw \|_{L^2(\Omega)}^2,
\quad \bv, \bw \in V.
\end{equation}
Moreover, we have the following lemma concerning $d_1$, whose proof is elementary; see, for instance,~\cite{Ciarlet:2002}.

\begin{lemma}
\label{Lem:d1}
The Bregman divergence $d_1$ given in~\eqref{d1} satisfies the following:
\begin{enumerate}[(a)]
\item $d_1(\bw; \bv) \lesssim \| \bw \|_{L^2(\Omega)}^2 + ( \| \bv \|_{L^3 (\Omega)} + \| \bw \|_{L^3 (\Omega)} )\| \bw \|_{L^3(\Omega)}^2$.
\item $d_1(\bw; \bv) \gtrsim \| \bw \|_{L^2(\Omega)}^2 + \| \bw \|_{L^3(\Omega)}^3$.
\end{enumerate}
\end{lemma}

In the abstract convergence theory presented in \cref{App:Nearly}, several assumptions are required to ensure the convergence of subspace correction methods; see \cref{Thm:nearly} for the convergence result and \cref{Ass:smooth,Ass:uniform,Ass:kernel,Ass:triangle} for the corresponding assumptions. In the following, we verify that \cref{Alg:multilevel} satisfies all these assumptions.

The local smoothness condition (\cref{Ass:smooth}) follows directly from~\eqref{d0} and \cref{Lem:d1}(a).
The local uniform convexity condition (\cref{Ass:uniform}) with $p = 3$ is similarly implied by~\eqref{d0} and \cref{Lem:d1}(b).
Namely, for any bounded and convex subset $K \subset V$, we have
\begin{equation}
\label{mu}
\mu_{0, K} := \inf_{\bv,\, \bv + \bw \in K} \frac{d_0 ( \bw; \bv )}{\| \div \bw \|_{L^2 (\Omega)}^3} \gtrsim 1,
\quad
\mu_{1, K} := \inf_{\bv,\, \bv + \bw \in K} \frac{(d_0 + d_1) (\bw; \bv)}{\| \bw \|_X^3} \gtrsim 1.
\end{equation}

The kernel decomposition condition (\cref{Ass:kernel}) is satisfied by the patch-based multilevel decomposition~\eqref{space_decomposition}, as discussed in~\cite[Section~5.1]{LWXZ:2007}. Finally, the triangle inequality-like property (\cref{Ass:triangle}) is a direct consequence of~\cite[Lemma~5.4]{LY:2001}; see also~\cite[Example~C.3]{LP:2025b}.

Given an initial guess $\bu^{(0)} \in V$, we define
\begin{equation}
\label{K_0}
K_0 := \{ \bv \in V : F^{\epsilon} ( \bv ) \leq F^{\epsilon} ( \bu^{(0)} ) \}.
\end{equation}
Since all assumptions in \cref{Thm:nearly} have been verified, we obtain the following convergence result for \cref{Alg:multilevel} by invoking that theorem.

\begin{theorem}
\label{Thm:convergence}
In \cref{Alg:multilevel}, assume that $\tau \in (0, \tau_0]$, where $\tau_0$ was given in~\eqref{tau_0}.
Then there exists $\zeta^* > 0$ such that, if $F^{\epsilon} (\bu^{(0)}) - F^{\epsilon} (\bu_h^{\epsilon}) > \zeta^*$, then
\begin{equation*}
F^{\epsilon} (\bu^{(1)}) - F^{\epsilon} (\bu_h^{\epsilon}) \leq \left(1 - \frac{\tau}{2} \right) ( F^{\epsilon} (\bu^{(0)}) - F^{\epsilon} (\bu_h^{\epsilon}) ),
\end{equation*}
and otherwise,
\begin{equation*}
F^{\epsilon} (\bu^{(n)}) - F^{\epsilon} (\bu_h^{\epsilon})
\leq \frac{ C_1^{3} }{ \tau^{3} C_2^{2} n^{3} },
\quad n \geq 1,
\end{equation*}
where $C_1$ is a positive constant independent of $\epsilon$, except for its implicit dependence on $K_0$, and $C_2 = \min \{\mu_{0, K_0}, \mu_{1, K_0} \}$~(see~\eqref{mu}).
\end{theorem}

We note that the convergence theorem above explains how the convergence rate of \cref{Alg:multilevel} depends on the augmented Lagrangian parameter $\epsilon$, but not explicitly on the mesh size $h$.
In fact, the constant $C_1$ in \cref{Thm:convergence}, which is defined in terms of stable decompositions of the multilevel subspaces~\eqref{space_decomposition} (see~\cite[Theorem~5.8]{LP:2025b}), may implicitly depend on $h$.

While $h$-independence of multilevel space decompositions for linear problems in $H(\div; \Omega)$ has been established in the literature~\cite{AFW:2000,HT:2000} using the Helmholtz decomposition, extending these arguments to our setting is highly nontrivial.
The main difficulty stems from the fact that our problem is posed on the space $X = H(\div; \Omega) \cap L^3(\Omega)^d$, which requires controlling the $L^3 (\Omega)^d$-norm in addition to the $H(\div; \Omega)$-norm.

In \cref{Sec:Numerical}, we provide numerical evidence indicating that the convergence of \cref{Alg:multilevel} is robust with respect to $h$, similarly to the case of linear problems, across a range of problem settings.

\begin{remark}
\label{Rem:linear_convergence}
Combining \cref{Lem:d1}(b) with the inverse inequality $\| \bw \|_{L^3(\Omega)} \lesssim h^{-\frac{d}{6}} \| \bw \|_{L^2(\Omega)}$~\cite[Theorem~4.5.11]{BS:2008}, we can verify that the local uniform convexity condition~(\cref{Ass:uniform}) holds with $p = 2$, although the constant $\mu_{1,K}$ depends on the mesh size $h$.
Consequently, by invoking \cref{Thm:nearly}(a), we obtain the linear convergence of \cref{Alg:multilevel}.
However, the resulting convergence rate is suboptimal due to its dependence on $h$, which originates from the inverse inequality.
\end{remark}

\begin{remark}
\label{Rem:Park:2024a}
In~\cite{Park:2024a}, an alternative approach was proposed for analyzing the $\epsilon$-independence of the convergence rate of subspace correction methods for solving problems of the form~\eqref{mixed_FEM_nearly}.
However, to apply the strategy developed in~\cite{Park:2024a}, one must construct a stable decomposition for the multilevel spaces~\eqref{space_decomposition} that satisfies certain specific conditions~\cite[Assumption~3.1]{Park:2024a}, which is not straightforward.
\end{remark}

\subsection{Backtracking line search}
As stated in \cref{Thm:convergence}, the step size $\tau$ must be chosen smaller than $\tau_0$ given in~\eqref{tau_0} to guarantee the convergence of \cref{Alg:multilevel}.
However, in this case, computing an explicit lower bound for $\tau_0$ is needed, and any such estimate may be overly conservative, potentially resulting in slow convergence of \cref{Alg:multilevel}.
To address this issue, we incorporate a backtracking line search strategy for selecting the step size, as introduced in~\cite{Park:2022a}, into the proposed multilevel method.
The resulting algorithm is summarized in \cref{Alg:multilevel_backt}.

\begin{algorithm}
\caption{Parallel multilevel method for solving~\eqref{mixed_FEM_nearly} with backtracking}
\begin{algorithmic}[]
\label{Alg:multilevel_backt}
\STATE Choose $\bu^{(0)} \in V$ and $\tau^{(0)} > 0$.
\FOR{$n=0,1,2,\dots$}
    \FOR{$j = 1, 2, \dots, J$}
        \FOR{$k = 1, 2, \dots, n_j$}
            \STATE $\displaystyle \bw_{j,k}^{(n+1)} = \operatornamewithlimits{\arg\min}_{\bw_{j,k} \in V_{j,k}} F^{\epsilon} ( \bu^{(n)} + \bw_{j,k} )
            $
        \ENDFOR
    \ENDFOR
    \STATE $\displaystyle \tau \leftarrow 2 \tau^{(n)}$
    \REPEAT
        \STATE $\displaystyle
        \bu^{(n+1)} = \bu^{(n)} + \tau \sum_{j=1}^J \sum_{k=1}^{n_j} \bw_{j,k}^{(n+1)}
        $
        \IF{$\displaystyle F^{\epsilon} (\bu^{(n)}) - F^{\epsilon} (\bu^{(n+1)}) < \tau \sum_{j=1}^J \sum_{k=1}^{n_j} \left( F^{\epsilon} (\bu^{(n)}) - F^{\epsilon} (\bu^{(n)} + \bw_{j,k}^{(n+1)}) \right)$}
            \STATE $\displaystyle \tau \leftarrow \frac{\tau}{2}$
        \ENDIF
    \UNTIL{$\displaystyle F^{\epsilon} (\bu^{(n)}) - F^{\epsilon} (\bu^{(n+1)}) \geq \tau \sum_{j=1}^J \sum_{k=1}^{n_j} \left( F^{\epsilon} (\bu^{(n)}) - F^{\epsilon} (\bu^{(n)} + \bw_{j,k}^{(n+1)}) \right)$}
    \STATE $\displaystyle \tau^{(n+1)} = \tau$
\ENDFOR
\end{algorithmic}
\end{algorithm}

The key idea behind the backtracking strategy is to adaptively select a step size $\tau$ such that the global energy decay exceeds the appropriately weighted sum of the local energy decays.
From the perspective of computational cost, at the $n$th outer iteration of \cref{Alg:multilevel_backt}, we need to evaluate $F^{\epsilon}(\bu^{(n)})$, $F^{\epsilon}(\bu^{(n)} + \bw_{j,k}^{(n+1)})$ for each $1 \le j \le J$ and $1 \le k \le n_j$, and $F^{\epsilon}(\bu^{(n+1)})$ for different trial values of the step size $\tau$.

Among these quantities, $F^{\epsilon}(\bu^{(n)})$ has already been computed in the previous iteration and can therefore be reused.
The values $F^{\epsilon}(\bu^{(n)} + \bw_{j,k}^{(n+1)})$ rely on the local updates $\bw_{j,k}^{(n+1)}$ and can thus be evaluated locally.
Consequently, the only global computation required is the evaluation of $F^{\epsilon}(\bu^{(n+1)})$, which must be repeated for different trial values of $\tau$.
However, the number of trials needed to determine $\tau^{(n+1)}$ is typically small when $\tau^{(n)}$ is chosen suitably~(cf.~\cite[Lemma~4]{Nesterov:2013}).
Hence, the additional computational cost of the backtracking procedure is modest.

It was shown in~\cite[Lemma~3.1]{Park:2022a} that the backtracking procedure always terminates, and a lower bound for $\tau^{(n)}$ can be established in terms of $\tau_0$ defined in~\eqref{tau_0}.
More precisely, from the strengthened convexity condition~\eqref{strengthened_convexity}, we have
\begin{equation*}
\begin{cases}
    \tau^{(n)} \geq \tau_0 & \text{ if } \tau^{(0)} = 2^m \tau_0 \text{ for some } m \in \mathbb{Z}, \\
    \tau^{(n)} > \frac{\tau_0}{2} & \text{ otherwise.}
\end{cases}
\end{equation*}
Therefore, \cref{Alg:multilevel_backt} does not require a priori knowledge of $\tau_0$.

Moreover, both theoretical and numerical results in~\cite{Park:2022a} demonstrate that the backtracking strategy not only eliminates the need to tune the step size $\tau$ but can also improve the convergence rate.
We do not give a separate convergence proof for \cref{Alg:multilevel_backt} in this paper.
This variant uses the same exact local solvers as \cref{Alg:multilevel}, but replaces the fixed step size by the backtracking strategy; its convergence is justified by the theory of parallel subspace correction methods with backtracking presented in~\cite{Park:2022a}.

\begin{remark}
\label{Rem:RSM}
While the backtracking line search introduced here is designed to satisfy the strengthened convexity condition~\eqref{strengthened_convexity}, enforcing this condition is essential because we use the overlapping space decomposition~\eqref{space_decomposition}.
Namely, since updates on the overlap regions are counted multiple times, the step size $\tau$ must be chosen sufficiently small so that the total contribution is not too large.

For linear problems, the restricted Schwarz method~\cite{CS:1999,EG:2003,FS:2001} is a successful approach for handling this multiple-counting issue.
Although overlapping decompositions are still used, only the correction supported on the nonoverlapping region is retained, which avoids multiple counting and often improves convergence.

In this sense, restricted Schwarz may serve as an alternative to the backtracking line search.
However, the restriction makes the algorithm nonsymmetric, which may complicate the analysis; recent analyses of additive Schwarz methods for convex optimization, such as~\cite{LP:2025b,Park:2020}, rely heavily on symmetry.
\end{remark}

\subsection{Full approximation scheme (FAS)}
At each iteration of either \cref{Alg:multilevel,Alg:multilevel_backt}, we need to solve local problems of the form
\begin{equation}
\label{local_exact}
\min_{\bw_{j,k} \in V_{j,k}} F^{\epsilon}(\bu^{\text{old}}+\bw_{j,k}),
\end{equation}
where $\bu^{\text{old}} \in V$ denotes the current iterate.
However, the problem~\eqref{local_exact} requires evaluating the energy $F^{\epsilon}$ on the finest scale at $\bu^{\text{old}}$, which may be computationally expensive~(cf.~\cite[Remark~4.2]{TX:2002}).

To further reduce the computational cost, we introduce the full approximation scheme~(FAS)~\cite{Brandt:1977,Brandt:1984,BF:2024,CHW:2020}, which replaces~\eqref{local_exact} with an approximate problem that only requires coarse-scale computations on $V_{j,k}$.
For completeness, we briefly derive the FAS formulation.

We start from the following variational identity, which follows directly from the optimality condition of~\eqref{mixed_FEM_nearly}:
\begin{equation}
\label{trivial}
\langle (F^{\epsilon})'(\bu_h^{\epsilon})-(F^{\epsilon})'(\bu^{\text{old}}), \bv \rangle
= -\langle (F^{\epsilon})'(\bu^{\text{old}}), \bv \rangle,
\quad \bv \in V.
\end{equation}
To reduce the problem~\eqref{trivial} to one posed on the $j$th-level subspace $V_j$ with an unknown $\tilde{\bu}_j \in V_j$, we replace $\bu_h^{\epsilon}$, $\bu^{\text{old}}$, and the test space $V$ in~\eqref{trivial} by $\tilde{\bu}_j$, $\Pi_j \bu^{\text{old}}$, and $V_j$, respectively.
Here, $\Pi_j \colon V \to V_j$ denotes the canonical interpolation operator~\cite{Brandt:1984,BF:2024}.
This leads to the following variational problem: find $\tilde{\bu}_j \in V_j$ such that
\begin{equation}
\label{reduced}
\langle (F^{\epsilon})'(\tilde{\bu}_j)-(F^{\epsilon})'(\Pi_j \bu^{\text{old}}), \bv_j \rangle
= -\langle (F^{\epsilon})'(\bu^{\text{old}}), \bv_j \rangle,
\quad \bv_j \in V_j.
\end{equation}

When solving~\eqref{reduced} by a patch-based Schwarz method~\cite{AFW:2000,LWC:2009}, we need to solve the following variational problem defined on $V_{j,k}$:
find $\tilde{\bw}_{j,k} \in V_{j,k}$ such that
\begin{equation}
\label{FAS}
\langle (F^{\epsilon})'(\Pi_j \bu^{\text{old}} + \tilde{\bw}_{j,k}) -(F^{\epsilon})'(\Pi_j \bu^{\text{old}}), \bv_{j,k} \rangle
= -\langle (F^{\epsilon})'(\bu^{\text{old}}), \bv_{j,k} \rangle,
\quad \bv_{j,k} \in V_{j,k}.
\end{equation}
Note that~\eqref{FAS} is equivalent to the following minimization problem:
\begin{equation}
\label{FAS_opt}
\tilde{\bw}_{j,k} = \operatornamewithlimits{\arg\min}_{\bw_{j,k} \in V_{j,k}}
\left\{
F^{\epsilon}(\Pi_j \bu^{\text{old}} + \bw_{j,k})
+\langle (F^{\epsilon})'(\bu^{\text{old}})-(F^{\epsilon})'(\Pi_j \bu^{\text{old}}), \bw_{j,k} \rangle
\right\}.
\end{equation}
Owing to the optimization structure in~\eqref{FAS_opt}, the FAS local problem~\eqref{FAS} can be interpreted as an inexact local solver in the sense of~\cite{LP:2025b,Park:2020}.
Thus, \cref{Alg:multilevel_FAS} is not covered directly by \cref{Thm:convergence}; its convergence justification relies on existing theory for inexact local solvers and FAS, such as~\cite{CHW:2020}.

In summary, \cref{Alg:multilevel_FAS} presents the recommended practical variant of the proposed parallel multilevel method for solving~\eqref{mixed_FEM_nearly}, equipped with both the backtracking line search and the FAS local problem~\eqref{FAS_opt}.

\begin{algorithm}
\caption{Parallel multilevel method for solving~\eqref{mixed_FEM_nearly} with FAS}
\begin{algorithmic}[]
\label{Alg:multilevel_FAS}
\STATE Choose $\bu^{(0)} \in V$ and $\tau^{(0)} > 0$.
\FOR{$n=0,1,2,\dots$}
    \FOR{$j = 1, 2, \dots, J$}
        \FOR{$k = 1, 2, \dots, n_j$}
            \STATE \resizebox{0.87\textwidth}{!}{$\displaystyle \bw_{j,k}^{(n+1)} = \operatornamewithlimits{\arg\min}_{\bw_{j,k} \in V_{j,k}}
\left\{
F^{\epsilon}(\Pi_j \bu^{(n)} + \bw_{j,k})
+\langle (F^{\epsilon})'(\bu^{(n)})-(F^{\epsilon})'(\Pi_j \bu^{(n)}), \bw_{j,k} \rangle
\right\}
            $}
        \ENDFOR
    \ENDFOR
    \STATE $\displaystyle \tau \leftarrow 2 \tau^{(n)}$
    \REPEAT
        \STATE $\displaystyle
        \bu^{(n+1)} = \bu^{(n)} + \tau \sum_{j=1}^J \sum_{k=1}^{n_j} \bw_{j,k}^{(n+1)}
        $
        \IF{$\displaystyle F^{\epsilon} (\bu^{(n)}) - F^{\epsilon} (\bu^{(n+1)}) < \tau \sum_{j=1}^J \sum_{k=1}^{n_j} \left( F^{\epsilon} (\bu^{(n)}) - F^{\epsilon} (\bu^{(n)} + \bw_{j,k}^{(n+1)}) \right)$}
            \STATE $\displaystyle \tau \leftarrow \frac{\tau}{2}$
        \ENDIF
    \UNTIL{$\displaystyle F^{\epsilon} (\bu^{(n)}) - F^{\epsilon} (\bu^{(n+1)}) \geq \tau \sum_{j=1}^J \sum_{k=1}^{n_j} \left( F^{\epsilon} (\bu^{(n)}) - F^{\epsilon} (\bu^{(n)} + \bw_{j,k}^{(n+1)}) \right)$}
    \STATE $\displaystyle \tau^{(n+1)} = \tau$
\ENDFOR
\end{algorithmic}
\end{algorithm}

\section{Numerical results}
\label{Sec:Numerical}
In this section, we present numerical results for the proposed multilevel methods~(\cref{Alg:multilevel,Alg:multilevel_backt,Alg:multilevel_FAS}) for solving the Darcy--Forchheimer model.
We also include results for the augmented Lagrangian method~(\cref{Alg:aug}) to support our theoretical findings.
All numerical experiments were conducted using MATLAB R2024b on a desktop equipped with an AMD Ryzen~5 5600X CPU (3.7~GHz, 6 cores), 40~GB RAM, and the Windows~10 Pro operating system.
The reported quantities related to parallel performance are implementation-independent operation-count proxies; we do not report wall-clock timings from a distributed-memory parallel implementation.

For the Darcy--Forchheimer model~\eqref{Forchheimer}, we set $\Omega = (0,1)^2 \subset \mathbb{R}^2$ and assume that $\bK = K \bI$ for some positive scalar $K$.
Following the settings commonly used in the literature on numerical methods for the Darcy--Forchheimer model, we take $\mu = 1$, $\rho = 1$, and $K = 1$ as in~\cite{LMS:2009,HCR:2018,PR:2012}, and vary the Forchheimer coefficient $\beta = 0, 10, 20, 30$, as in~\cite{HCR:2018,LMS:2009}.
Note that the case $\beta = 0$ corresponds to the linear Darcy model.

We consider the following two benchmark examples taken from~\cite{PR:2012}.

\begin{example}
\label{Ex:Ex1}
The functions $\bf$ and $g$, and the corresponding exact solutions $\bu$ and $p$ of~\eqref{Forchheimer} are given by
\begin{align*}
\bf (x,y) &= \left(\frac{\mu}{\rho K} + \frac{\beta}{\rho} e^x \right)(e^x \sin y, e^x \cos y)^T + (y(1-2x)(1-y), x(1-x)(1 - 2y))^T, \\
g (x,y) &= 0, \quad
\bu (x,y) = ( e^x \sin y, e^x \cos y)^T, \quad
p (x,y) = xy(1-x)(1-y).
\end{align*}
\end{example}

\begin{example}
\label{Ex:Ex2}
The functions $\bf$ and $g$, and the corresponding exact solutions $\bu$ and $p$ of~\eqref{Forchheimer} are given by
\begin{align*}
\bf (x,y) &= \left(\frac{\mu}{\rho K} + \frac{\beta}{\rho} \sqrt{x^2 e^{2y} + y^2 e^{2x}}\right)(x e^y, y e^x)^T + \pi ( \cos \pi x \sin \pi y, \sin \pi x \cos \pi y)^T, \\
g (x,y) &= e^x + e^y, \quad
\bu (x,y) = (x e^y, y e^x )^T, \quad
p (x,y) = \sin \pi x \sin \pi y.
\end{align*}
\end{example}

For the finite element discretization of the Darcy--Forchheimer model~\eqref{Forchheimer}, we use the lowest-order rectangular Raviart--Thomas-type conforming mixed finite element space~(see~\eqref{RT_mixed}) defined on a uniform $1/h \times 1/h$ grid, where $h$ denotes the side length of each rectangular element.
For the multilevel space decomposition~\eqref{space_decomposition} of the finite element space, we use the coarsening factor $\gamma = 2^{-1}$ unless stated otherwise.
Namely, the number of levels $J$ satisfies $h = 2^{-J}$.

Each local problem in the proposed multilevel methods is solved by the damped Newton method~\cite{BV:2004}, which is known to be globally convergent for convex optimization problems.
The stopping criterion is given in terms of the relative energy difference:
\begin{equation}
\label{stop_local}
\left| \frac{F_{j,k}^{\epsilon} (\bw_{j,k}^{(\ell+1)} ; \bu^{\text{old}}) - F_{j,k}^{\epsilon} (\bw_{j,k}^{(\ell)} ; \bu^{\text{old}}) }{ F_{j,k}^{\epsilon} (\bw_{j,k}^{(\ell+1)} ; \bu^{\text{old}}) } \right| < 10^{-5},
\end{equation}
where $\bw_{j,k}^{(\ell)}$ denotes the $\ell$th local Newton iterate, $\bu^{\text{old}}$ denotes the current global iterate, and $F_{j,k}^{\epsilon}(\cdot; \bu^{\text{old}}) \colon V_{j,k} \to \mathbb{R}$ denotes the corresponding local energy functional.

In all experiments and for all algorithms, we use zero initial guesses.
For the nearly semicoercive problem~\eqref{mixed_FEM_nearly}, we set $q_h = 0$ throughout.
In the linear case~($\beta = 0$), we employ \cref{Alg:multilevel} as a preconditioner for the conjugate gradient method~\cite{BPX:1990}.

\subsection{Augmented Lagrangian method}

\begin{table}
\label{Table:aug_Ex1}
\centering
\caption{Number of iterations required by the augmented Lagrangian method~(\cref{Alg:aug}) to satisfy the stopping criterion~\eqref{stop_aug}, for various values of the augmented Lagrangian parameter $\epsilon$ and mesh size $h$~(\cref{Ex:Ex1}).}
\begin{tabular}{c|cccc c c|cccc}
\multicolumn{5}{c}{$\beta = 0$} & \quad \quad & \multicolumn{5}{c}{$\beta = 10$} \\
\cline{1-5} \cline{7-11}
\diagbox{$h$}{$\epsilon$} & $10^0$ & $10^{-1}$ & $10^{-2}$ & $10^{-3}$ && \diagbox{$h$}{$\epsilon$} & $10^0$ & $10^{-1}$ & $10^{-2}$ & $10^{-3}$ \\
\cline{1-5} \cline{7-11}
$2^{-6}$ & 3 & 2 & 1 & 1 && $2^{-6}$ & 12 & 4 & 2 & 2 \\
$2^{-7}$ & 3 & 2 & 1 & 1 && $2^{-7}$ & 12 & 4 & 2 & 2 \\
$2^{-8}$ & 3 & 2 & 1 & 1 && $2^{-8}$ & 12 & 4 & 2 & 2 \\
$2^{-9}$ & 3 & 2 & 1 & 1 && $2^{-9}$ & 12 & 4 & 2 & 2 \\
\cline{1-5} \cline{7-11}
\multicolumn{11}{c}{} \\
\multicolumn{5}{c}{$\beta = 20$} & \quad \quad & \multicolumn{5}{c}{$\beta = 30$} \\
\cline{1-5} \cline{7-11}
\diagbox{$h$}{$\epsilon$} & $10^0$ & $10^{-1}$ & $10^{-2}$ & $10^{-3}$ && \diagbox{$h$}{$\epsilon$} & $10^0$ & $10^{-1}$ & $10^{-2}$ & $10^{-3}$ \\
\cline{1-5} \cline{7-11}
$2^{-6}$ & 19 & 5 & 2 & 2 && $2^{-6}$ & 26 & 6 & 3 & 2 \\
$2^{-7}$ & 20 & 5 & 2 & 2 && $2^{-7}$ & 27 & 6 & 3 & 2 \\
$2^{-8}$ & 19 & 5 & 2 & 2 && $2^{-8}$ & 27 & 6 & 3 & 2 \\
$2^{-9}$ & 20 & 5 & 2 & 2 && $2^{-9}$ & 27 & 6 & 3 & 2 \\
\cline{1-5} \cline{7-11}
\end{tabular}
\end{table}

\begin{table}
\label{Table:aug_Ex2}
\centering
\caption{Number of iterations required by the augmented Lagrangian method~(\cref{Alg:aug}) to satisfy the stopping criterion~\eqref{stop_aug}, for various values of the augmented Lagrangian parameter $\epsilon$ and mesh size $h$~(\cref{Ex:Ex2}).}
\begin{tabular}{c|cccc c c|cccc}
\multicolumn{5}{c}{$\beta = 0$} & \quad \quad & \multicolumn{5}{c}{$\beta = 10$} \\
\cline{1-5} \cline{7-11}
\diagbox{$h$}{$\epsilon$} & $10^0$ & $10^{-1}$ & $10^{-2}$ & $10^{-3}$ && \diagbox{$h$}{$\epsilon$} & $10^0$ & $10^{-1}$ & $10^{-2}$ & $10^{-3}$ \\
\cline{1-5} \cline{7-11}
$2^{-6}$ & 3 & 2 & 1 & 1 && $2^{-6}$ & 9 & 3 & 2 & 1 \\
$2^{-7}$ & 3 & 2 & 1 & 1 && $2^{-7}$ & 9 & 3 & 2 & 1 \\
$2^{-8}$ & 3 & 2 & 1 & 1 && $2^{-8}$ & 9 & 3 & 2 & 1 \\
$2^{-9}$ & 3 & 2 & 1 & 1 && $2^{-9}$ & 9 & 3 & 2 & 1 \\
\cline{1-5} \cline{7-11}
\multicolumn{11}{c}{} \\
\multicolumn{5}{c}{$\beta = 20$} & \quad \quad & \multicolumn{5}{c}{$\beta = 30$} \\
\cline{1-5} \cline{7-11}
\diagbox{$h$}{$\epsilon$} & $10^0$ & $10^{-1}$ & $10^{-2}$ & $10^{-3}$ && \diagbox{$h$}{$\epsilon$} & $10^0$ & $10^{-1}$ & $10^{-2}$ & $10^{-3}$ \\
\cline{1-5} \cline{7-11}
$2^{-6}$ & 15 & 4 & 2 & 2 && $2^{-6}$ & 20 & 5 & 2 & 2 \\
$2^{-7}$ & 15 & 4 & 2 & 2 && $2^{-7}$ & 20 & 5 & 2 & 2 \\
$2^{-8}$ & 15 & 4 & 2 & 2 && $2^{-8}$ & 20 & 5 & 2 & 2 \\
$2^{-9}$ & 15 & 4 & 2 & 2 && $2^{-9}$ & 20 & 5 & 2 & 2 \\
\cline{1-5} \cline{7-11}
\end{tabular}
\end{table}

Numerical results for the augmented Lagrangian method~(\cref{Alg:aug}) for various values of the augmented Lagrangian parameter $\epsilon$ and mesh size $h$ for solving \cref{Ex:Ex1,Ex:Ex2} are presented in \cref{Table:aug_Ex1,Table:aug_Ex2}, respectively.
For the outer iteration, we use the stopping criterion
\begin{equation}
\label{stop_aug}
\max \left\{ \frac{\| \bu_h^{(n)} - \bu_h \|_{\ell^2} }{\| \bu_h \|_{\ell^2}}, \frac{\| p_h^{(n)} - p_h \|_{\ell^2} }{\| p_h \|_{\ell^2}} \right\} < 10^{-3},
\end{equation}
where $\| \cdot \|_{\ell^2}$ denotes the standard $\ell^2$-norm of the vector of degrees of freedom.

To numerically verify the arbitrarily fast convergence of \cref{Alg:aug} stated in \cref{Thm:aug}, we solve each subproblem using a globally convergent solver from the literature.
Specifically, each subproblem in \cref{Alg:aug} is solved by the damped Newton method~\cite{BV:2004}, using the same type of relative-energy stopping criterion as in~\eqref{stop_local}, with tolerance $10^{-5}$.

In all cases, we observe that the number of iterations decreases as $\epsilon$ decreases, in agreement with \cref{Thm:aug}.
Moreover, the number of iterations remains uniformly bounded as the mesh is refined.
Hence, we conclude that \cref{Alg:aug} converges in very few iterations even for small values of $h$, provided that $\epsilon$ is sufficiently small.
These results support our claim that the mixed problem~\eqref{mixed_FEM_saddle} can be reduced to the optimization problem~\eqref{mixed_FEM_nearly}.

\subsection{Backtracking line search and FAS}

\begin{figure}
    \centering
    \includegraphics[width=\textwidth]{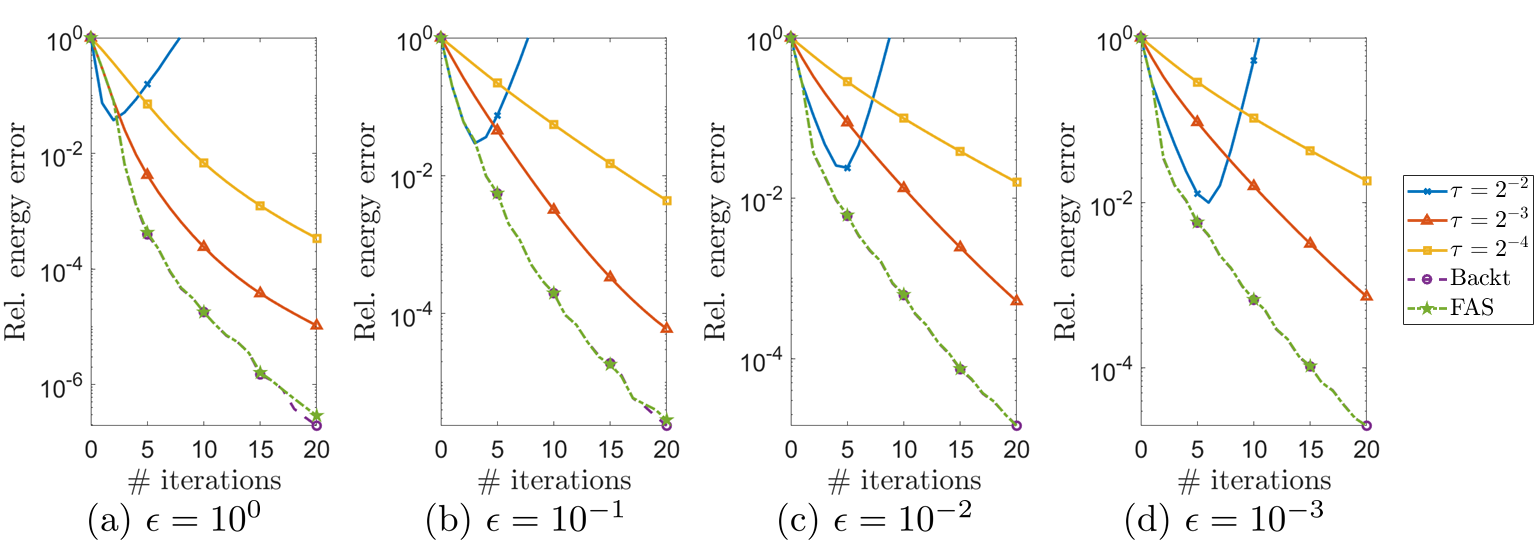}
    \caption{Convergence curves of the relative energy error $\frac{F^{\epsilon}(\bu^{(n)}) - \min F^{\epsilon}}{|\min F^{\epsilon}|}$ for the multilevel method~(\cref{Alg:multilevel}) with step sizes $\tau = 2^{-2}, 2^{-3}, 2^{-4}$, the backtracking variant~(\cref{Alg:multilevel_backt}), and the FAS variant~(\cref{Alg:multilevel_FAS}), for various values of the augmented Lagrangian parameter $\epsilon$. 
    Solid, dashed, and dash--dot lines correspond to Algorithms~\ref{Alg:multilevel}, \ref{Alg:multilevel_backt}, and \ref{Alg:multilevel_FAS}, respectively (\cref{Ex:Ex1}, $\beta = 30$, $h = 2^{-8}$).}
    \label{Fig:backt_Ex1}
\end{figure}

\begin{figure}
    \centering
    \includegraphics[width=\textwidth]{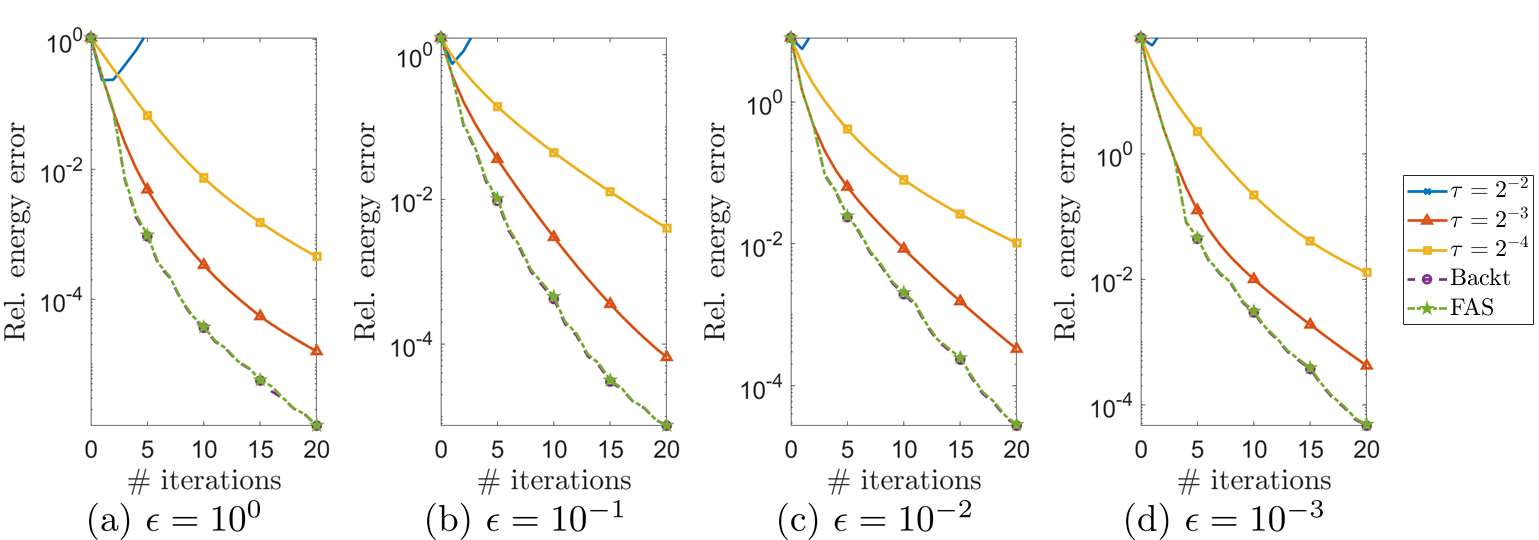}
    \caption{Convergence curves of the relative energy error $\frac{F^{\epsilon}(\bu^{(n)}) - \min F^{\epsilon}}{|\min F^{\epsilon}|}$ for the multilevel method~(\cref{Alg:multilevel}) with step sizes $\tau = 2^{-2}, 2^{-3}, 2^{-4}$, the backtracking variant~(\cref{Alg:multilevel_backt}), and the FAS variant~(\cref{Alg:multilevel_FAS}), for various values of the augmented Lagrangian parameter $\epsilon$. 
    Solid, dashed, and dash--dot lines correspond to Algorithms~\ref{Alg:multilevel}, \ref{Alg:multilevel_backt}, and \ref{Alg:multilevel_FAS}, respectively (\cref{Ex:Ex2}, $\beta = 30$, $h = 2^{-8}$).}
    \label{Fig:backt_Ex2}
\end{figure}

\begin{table}
\label{Table:backt}
\centering
\caption{Minimum step size $\min_{n \geq 0} \tau^{(n)}$ in the proposed multilevel methods with backtracking~(\cref{Alg:multilevel_backt,Alg:multilevel_FAS}) in the numerical experiments corresponding to \cref{Fig:backt_Ex1,Fig:backt_Ex2}.}
\begin{tabular}{c|cccc}
\cline{1-5}
& $\epsilon = 10^0$ & $\epsilon = 10^{-1}$ & $\epsilon = 10^{-2}$ & $\epsilon = 10^{-3}$ \\
\cline{1-5}
\cref{Alg:multilevel_backt} & $2^{-3}$ & $2^{-3}$ & $2^{-3}$ & $2^{-3}$ \\
\cref{Alg:multilevel_FAS} & $2^{-3}$ & $2^{-3}$ & $2^{-3}$ & $2^{-3}$ \\
\cline{1-5}
\multicolumn{4}{c}{(a) \cref{Ex:Ex1}} \\
\multicolumn{5}{c}{} \\
\cline{1-5}
& $\epsilon = 10^0$ & $\epsilon = 10^{-1}$ & $\epsilon = 10^{-2}$ & $\epsilon = 10^{-3}$ \\
\cline{1-5}
\cref{Alg:multilevel_backt} & $2^{-3}$ & $2^{-3}$ & $2^{-4}$ & $2^{-4}$ \\
\cref{Alg:multilevel_FAS} & $2^{-3}$ & $2^{-3}$ & $2^{-4}$ & $2^{-4}$ \\
\cline{1-5}
\multicolumn{4}{c}{(b) \cref{Ex:Ex2}}
\end{tabular}
\end{table}

To observe the effects of the backtracking strategy and FAS introduced in \cref{Alg:multilevel_backt,Alg:multilevel_FAS}, we present convergence curves for the multilevel methods \cref{Alg:multilevel,Alg:multilevel_backt,Alg:multilevel_FAS} in \cref{Fig:backt_Ex1,Fig:backt_Ex2}, where step sizes $\tau = 2^{-2}, 2^{-3}, 2^{-4}$ are used for \cref{Alg:multilevel}.
For \cref{Alg:multilevel}, we observe that when the step size is set to $\tau = 2^{-2}$, the algorithm diverges in almost all cases, indicating that overly large step sizes may lead to divergence.  
For smaller step sizes $\tau \leq 2^{-3}$, the algorithm converges in all cases, and larger step sizes among them result in faster convergence.  
This implies that, provided $\tau$ is sufficiently small to ensure convergence, larger values within the admissible range yield better performance.

In contrast, for \cref{Alg:multilevel_backt}, which adaptively determines the step size using the backtracking scheme, we consistently observe faster convergence rates across all test cases when compared to \cref{Alg:multilevel}.  
This demonstrates the effectiveness of the backtracking strategy in improving convergence behavior.

Finally, we observe that the convergence curves of \cref{Alg:multilevel_backt,Alg:multilevel_FAS} are nearly indistinguishable in all cases.
This indicates that using the FAS local solver~\eqref{FAS_opt} in the proposed multilevel method, as in \cref{Alg:multilevel_FAS}, does not deteriorate the convergence properties of the algorithm, while significantly reducing the computational cost.

Based on these observations, we focus exclusively on \cref{Alg:multilevel_FAS} in the remaining numerical experiments, as it achieves essentially the same convergence behavior as \cref{Alg:multilevel_backt} while reducing the local computational cost.

We also report the minimum step sizes $\min_{n \geq 0} \tau^{(n)}$ of \cref{Alg:multilevel_backt,Alg:multilevel_FAS} in \cref{Table:backt}.
We observe that the minimum step size does not decrease significantly as the augmented Lagrangian parameter $\epsilon$ becomes small.
This implies that the backtracking scheme is robust with respect to the nearly semicoercive behavior of the problem~\eqref{mixed_FEM_nearly}.

\subsection{Dependence on \texorpdfstring{$\epsilon$, $h$}{epsilon and h}, and coarsening factor}

\begin{table}
\label{Table:dependence_Ex1}
\centering
\caption{Numerical results of the proposed multilevel method (\cref{Alg:multilevel_FAS}, coarsening factor $\gamma = 2^{-1}$) satisfying the stopping criterion~\eqref{stop_multilevel}, for various values of the augmented Lagrangian parameter $\epsilon$ and mesh size $h$ (\cref{Ex:Ex1}).}
\resizebox{0.99\textwidth}{!}{
\begin{tabular}{c|cccc|cccc|cccc}
\multicolumn{13}{c}{$\beta = 10$} \\
\cline{1-13}
 & \multicolumn{4}{c|}{outer iterations} &
 \multicolumn{4}{c|}{local Hessian assemblies/solves} & \multicolumn{4}{c}{parallel patch inner solves} \\
\cline{2-13}
\diagbox{$h$}{$\epsilon$} & 
$10^0$ & $10^{-1}$ & $10^{-2}$ & $10^{-3}$ &
$10^0$ & $10^{-1}$ & $10^{-2}$ & $10^{-3}$ & 
$10^0$ & $10^{-1}$ & $10^{-2}$ & $10^{-3}$ \\
\cline{1-13}
$2^{-6}$ & 
5 & 9 & 9 & 9 &     
1.1E5 & 1.6E5 & 1.5E5 & 1.4E5 &  
41 & 75 & 69 & 69 \\
$2^{-7}$ & 
5 & 9 & 9 & 9 &     
4.5E5 & 6.2E5 & 5.7E5 & 5.4E5 &
43 & 78 & 72 & 72 \\
$2^{-8}$ &  
6 & 10 & 10 & 10 &      
1.8E6 & 2.4E6 & 2.2E6 & 2.1E6 & 
49 & 90 & 83 & 84 \\
$2^{-9}$ &  
6 & 10 & 10 & 10 &     
6.4E6 & 8.4E6 & 8.0E6 & 7.7E6 &
53 & 93 & 93 & 87 \\
\cline{1-13}
& \multicolumn{4}{c|}{global reductions} & \multicolumn{4}{c|}{global energy evaluations} & \multicolumn{4}{c}{global gradient evaluations} \\
\cline{2-13}
\diagbox{$h$}{$\epsilon$} & 
$10^0$ & $10^{-1}$ & $10^{-2}$ & $10^{-3}$ &
$10^0$ & $10^{-1}$ & $10^{-2}$ & $10^{-3}$ & 
$10^0$ & $10^{-1}$ & $10^{-2}$ & $10^{-3}$ \\
\cline{1-13}
$2^{-6}$ & 
5 & 9 & 9 & 9 &     
11 & 19 & 20 & 19 &   
5 & 9 & 9 & 9 \\
$2^{-7}$ &
5 & 9 & 9 & 9 &     
12 & 19 & 19 & 20 &   
5 & 9 & 9 & 9 \\
$2^{-8}$ &  
6 & 10 & 10 & 10 &      
13 & 21 & 21 & 22 &    
6 & 10 & 10 & 10 \\
$2^{-9}$ & 
6 & 10 & 10 & 10 &     
14 & 21 & 22 & 23 &   
6 & 10 & 10 & 10 \\
\cline{1-13}
\multicolumn{13}{c}{} \\
\multicolumn{13}{c}{$\beta = 20$} \\
\cline{1-13}
 & \multicolumn{4}{c|}{outer iterations} &
 \multicolumn{4}{c|}{local Hessian assemblies/solves} & \multicolumn{4}{c}{parallel patch inner solves} \\
\cline{2-13}
\diagbox{$h$}{$\epsilon$} & 
$10^0$ & $10^{-1}$ & $10^{-2}$ & $10^{-3}$ &
$10^0$ & $10^{-1}$ & $10^{-2}$ & $10^{-3}$ & 
$10^0$ & $10^{-1}$ & $10^{-2}$ & $10^{-3}$ \\
\cline{1-13}
$2^{-6}$ & 
4 & 8 & 8 & 9 &     
8.1E4 & 1.5E5 & 1.4E5 & 1.4E5 &
34 & 69 & 65 & 70 \\
$2^{-7}$ & 
5 & 8 & 9 & 9 &      
3.5E5 & 5.6E5 & 5.7E5 & 5.3E5 &
39 & 73 & 76 & 72 \\
$2^{-8}$ & 
5 & 8 & 10 & 9 &      
1.3E6 & 2.0E6 & 2.1E6 & 1.9E6 &
42 & 73 & 87 & 76 \\
$2^{-9}$ & 
5 & 9 & 10 & 10 &      
5.5E6 & 7.7E6 & 7.9E6 & 7.6E6 &
43 & 86 & 92 & 90 \\
\cline{1-13}
& \multicolumn{4}{c|}{global reductions} & \multicolumn{4}{c|}{global energy evaluations} & \multicolumn{4}{c}{global gradient evaluations} \\
\cline{2-13}
\diagbox{$h$}{$\epsilon$} & 
$10^0$ & $10^{-1}$ & $10^{-2}$ & $10^{-3}$ &
$10^0$ & $10^{-1}$ & $10^{-2}$ & $10^{-3}$ & 
$10^0$ & $10^{-1}$ & $10^{-2}$ & $10^{-3}$ \\
\cline{1-13}
$2^{-6}$ & 
4 & 8 & 8 & 9 &     
9 & 17 & 17 & 19 &   
4 & 8 & 8 & 9 \\
$2^{-7}$ &  
5 & 8 & 9 & 9 &      
11 & 18 & 19 & 20 &    
5 & 8 & 9 & 9 \\
$2^{-8}$ & 
5 & 8 & 10 & 9 &      
11 & 17 & 22 & 19 &    
5 & 8 & 10 & 9 \\
$2^{-9}$ &  
5 & 9 & 10 & 10 &      
11 & 19 & 21 & 21 &    
5 & 9 & 10 & 10 \\
\cline{1-13}
\multicolumn{13}{c}{} \\
\multicolumn{13}{c}{$\beta = 30$} \\
\cline{1-13}
 & \multicolumn{4}{c|}{outer iterations} &
 \multicolumn{4}{c|}{local Hessian assemblies/solves} & \multicolumn{4}{c}{parallel patch inner solves} \\
\cline{2-13}
\diagbox{$h$}{$\epsilon$} & 
$10^0$ & $10^{-1}$ & $10^{-2}$ & $10^{-3}$ &
$10^0$ & $10^{-1}$ & $10^{-2}$ & $10^{-3}$ & 
$10^0$ & $10^{-1}$ & $10^{-2}$ & $10^{-3}$ \\
\cline{1-13}
$2^{-6}$ &
5 & 7 & 8 & 9 &     
1.3E5 & 1.5E5 & 1.5E5 & 1.5E5 &
48 & 62 & 66 & 72 \\
$2^{-7}$ & 
5& 7 & 9 & 9 &      
4.4E5 & 5.2E5 & 5.6E5 & 5.4E5 &
49 & 63 & 75 & 74 \\
$2^{-8}$ &
5 & 8 & 9 & 9 &      
1.5E6 & 1.9E6 & 2.0E6 & 1.9E6 &
47 & 73 & 81 & 76 \\
$2^{-9}$ & 
5 & 8 & 10 & 10 &      
5.8E6 & 6.9E6 & 7.7E6 & 7.4E6 &
47 & 77 & 94 & 89 \\
\cline{1-13}
& \multicolumn{4}{c|}{global reductions} & \multicolumn{4}{c|}{global energy evaluations} & \multicolumn{4}{c}{global gradient evaluations} \\
\cline{2-13}
\diagbox{$h$}{$\epsilon$} & 
$10^0$ & $10^{-1}$ & $10^{-2}$ & $10^{-3}$ &
$10^0$ & $10^{-1}$ & $10^{-2}$ & $10^{-3}$ & 
$10^0$ & $10^{-1}$ & $10^{-2}$ & $10^{-3}$ \\
\cline{1-13}
$2^{-6}$ & 
5 & 7 & 8 & 9 &     
11 & 16 & 17 & 19 &   
5 & 7 & 8 & 9 \\
$2^{-7}$ & 
5 & 7 & 9 & 9 &      
11 & 15 & 20 & 20 &    
5 & 7 & 9 & 9 \\
$2^{-8}$ &  
5 & 8 & 9 & 9 &      
11 & 17 & 19 & 19 &    
5 & 8 & 9 & 9 \\
$2^{-9}$ &  
5 & 8 & 10 & 10 &      
12 & 18 & 22 & 21 &    
5 & 8 & 10 & 10 \\
\cline{1-13}
\end{tabular}
}
\end{table}

\begin{table}
\label{Table:dependence_Ex2}
\centering
\caption{Numerical results of the proposed multilevel method (\cref{Alg:multilevel_FAS}, coarsening factor $\gamma = 2^{-1}$) satisfying the stopping criterion~\eqref{stop_multilevel}, for various values of the augmented Lagrangian parameter $\epsilon$ and mesh size $h$ (\cref{Ex:Ex2}).}
\resizebox{0.99\textwidth}{!}{
\begin{tabular}{c|cccc|cccc|cccc}
\multicolumn{13}{c}{$\beta = 10$} \\
\cline{1-13}
 & \multicolumn{4}{c|}{outer iterations} &
 \multicolumn{4}{c|}{local Hessian assemblies/solves} & \multicolumn{4}{c}{parallel patch inner solves} \\
\cline{2-13}
\diagbox{$h$}{$\epsilon$} & 
$10^0$ & $10^{-1}$ & $10^{-2}$ & $10^{-3}$ &
$10^0$ & $10^{-1}$ & $10^{-2}$ & $10^{-3}$ & 
$10^0$ & $10^{-1}$ & $10^{-2}$ & $10^{-3}$ \\
\cline{1-13}
$2^{-6}$ & 
6 & 10 & 11 & 11 &     
1.3E5 & 1.6E5 & 1.8E5 & 1.9E5 &
61 & 105 & 117 & 124 \\
$2^{-7}$ & 
6 & 10 & 12 & 12 &     
4.5E5 & 6.0E5 & 7.0E5 & 7.7E5 &
64 & 113 & 139 & 149 \\
$2^{-8}$ & 
6 & 11 & 13 & 13 &      
1.5E6 & 2.2E6 & 2.7E6 & 2.9E6 &
65 & 124 & 154 & 170 \\
$2^{-9}$ & 
7 & 11 & 13 & 14 &      
5.3E6 & 7.6E6 & 9.5E6 & 1.1E7 &
78 & 128 & 153 & 177 \\
\cline{1-13}
& \multicolumn{4}{c|}{global reductions} & \multicolumn{4}{c|}{global energy evaluations} & \multicolumn{4}{c}{global gradient evaluations} \\
\cline{2-13}
\diagbox{$h$}{$\epsilon$} & 
$10^0$ & $10^{-1}$ & $10^{-2}$ & $10^{-3}$ &
$10^0$ & $10^{-1}$ & $10^{-2}$ & $10^{-3}$ & 
$10^0$ & $10^{-1}$ & $10^{-2}$ & $10^{-3}$ \\
\cline{1-13}
$2^{-6}$ & 
6 & 10 & 11 & 11 &     
13 & 21 & 24 & 23 &   
6 & 10 & 11 & 11 \\
$2^{-7}$ &
6 & 10 & 12 & 12 &      
14 & 22 & 25 & 25 &    
6 & 10 & 12 & 12 \\
$2^{-8}$ &  
6 & 11 & 13 & 13 &      
13 & 25 & 27 & 27 &    
6 & 11 & 13 & 13 \\
$2^{-9}$ &  
7 & 11 & 13 & 14 &      
15 & 24 & 27 & 30 &    
7 & 11 & 13 & 14 \\
\cline{1-13}
\multicolumn{13}{c}{} \\
\multicolumn{13}{c}{$\beta = 20$} \\
\cline{1-13}
 & \multicolumn{4}{c|}{outer iterations} &
 \multicolumn{4}{c|}{local Hessian assemblies/solves} & \multicolumn{4}{c}{parallel patch inner solves} \\
\cline{2-13}
\diagbox{$h$}{$\epsilon$} & 
$10^0$ & $10^{-1}$ & $10^{-2}$ & $10^{-3}$ &
$10^0$ & $10^{-1}$ & $10^{-2}$ & $10^{-3}$ & 
$10^0$ & $10^{-1}$ & $10^{-2}$ & $10^{-3}$ \\
\cline{1-13}
$2^{-6}$ & 
5 & 8 & 10 & 11 &     
1.1E5 & 1.5E5 & 1.7E5 & 1.9E5 &
52 & 87 & 112 & 127 \\
$2^{-7}$ &  
5 & 9 & 11 & 12 &      
3.9E5 & 5.7E5 & 6.7E5 & 7.6E5 &
54 & 104 & 129 & 149 \\
$2^{-8}$ &  
6 & 10 & 12 & 13 &      
1.5E6 & 2.1E6 & 2.5E6 & 2.9E6 &
65 & 116 & 158 & 171 \\
$2^{-9}$ &  
6 & 10 & 13 & 14 &      
4.9E6 & 6.9E6 & 9.2E6 & 1.1E7 &
70 & 116 & 158 & 171 \\
\cline{1-13}
& \multicolumn{4}{c|}{global reductions} & \multicolumn{4}{c|}{global energy evaluations} & \multicolumn{4}{c}{global gradient evaluations} \\
\cline{2-13}
\diagbox{$h$}{$\epsilon$} & 
$10^0$ & $10^{-1}$ & $10^{-2}$ & $10^{-3}$ &
$10^0$ & $10^{-1}$ & $10^{-2}$ & $10^{-3}$ & 
$10^0$ & $10^{-1}$ & $10^{-2}$ & $10^{-3}$ \\
\cline{1-13}
$2^{-6}$ & 
5 & 8 & 10 & 11 &   
11 & 17 & 21 & 23 &  
5 & 8 & 10 & 11 \\
$2^{-7}$ & 
5 & 9 & 11 & 12 &     
11& 19 & 24 & 25 &   
5 & 9 & 11 & 12 \\
$2^{-8}$ &  
6 & 10 & 12 & 13 & 
13 & 22 & 25 & 27 &
6 & 10 & 12 & 13 \\
$2^{-9}$ & 
6 & 10 & 13 & 14 &    
13 & 21 & 27 & 30 &   
6 & 10 & 13 & 14 \\
\cline{1-13}
\multicolumn{13}{c}{} \\
\multicolumn{13}{c}{$\beta = 30$} \\
\cline{1-13}
 & \multicolumn{4}{c|}{outer iterations} &
 \multicolumn{4}{c|}{local Hessian assemblies/solves} & \multicolumn{4}{c}{parallel patch inner solves} \\
\cline{2-13}
\diagbox{$h$}{$\epsilon$} & 
$10^0$ & $10^{-1}$ & $10^{-2}$ & $10^{-3}$ &
$10^0$ & $10^{-1}$ & $10^{-2}$ & $10^{-3}$ & 
$10^0$ & $10^{-1}$ & $10^{-2}$ & $10^{-3}$ \\
\cline{1-13}
$2^{-6}$ & 
5 & 8 & 10 & 11 &   
1.2E5 & 1.5E5 & 1.7E5 & 1.9E5 &
51 & 87 & 110 & 122 \\
$2^{-7}$ & 
5 & 8 & 11 & 12 &     
4.1E5 & 5.3E5 & 6.6E5 & 7.6E5 &
53 & 93 & 130 & 155 \\
$2^{-8}$ & 
5 & 9 & 12 & 13 &    
1.4E6 & 2.0E6 & 2.5E6 & 2.8E6 &
56 & 106 & 144 & 164 \\
$2^{-9}$ &  
5 & 9 & 12 & 14 &     
4.6E6 & 6.4E6 & 8.6E6 & 1.1E7 &
58 & 107 & 145 & 173 \\
\cline{1-13}
& \multicolumn{4}{c|}{global reductions} & \multicolumn{4}{c|}{global energy evaluations} & \multicolumn{4}{c}{global gradient evaluations} \\
\cline{2-13}
\diagbox{$h$}{$\epsilon$} & 
$10^0$ & $10^{-1}$ & $10^{-2}$ & $10^{-3}$ &
$10^0$ & $10^{-1}$ & $10^{-2}$ & $10^{-3}$ & 
$10^0$ & $10^{-1}$ & $10^{-2}$ & $10^{-3}$ \\
\cline{1-13}
$2^{-6}$ & 
5 & 8 & 10 & 11 &    
11 & 17 & 21 & 23 &  
5 & 8 & 10 & 11 \\
$2^{-7}$ & 
5 & 8 & 11 & 12 &    
11 & 18 & 24 & 26 &  
5 & 8 & 11 & 12 \\
$2^{-8}$ &  
5 & 9 & 12 & 13 &    
12 & 21 & 25 & 27 &  
5 & 9 & 12 & 13 \\
$2^{-9}$ & 
5 & 9 & 12 & 14 &   
11 & 20 & 26 & 29 &   
5 & 9 & 12 & 14 \\
\cline{1-13}
\end{tabular}
}
\end{table}

\begin{table}
\label{Table:dependence_Ex1_coarsening}
\centering
\caption{Numerical results of the proposed multilevel method (\cref{Alg:multilevel_FAS}, coarsening factor $\gamma = 2^{-2}$) satisfying the stopping criterion~\eqref{stop_multilevel}, for various values of the augmented Lagrangian parameter $\epsilon$ and mesh size $h$ (\cref{Ex:Ex1}).}
\resizebox{0.99\textwidth}{!}{
\begin{tabular}{c|cccc|cccc|cccc}
\multicolumn{13}{c}{$\beta = 10$} \\
\cline{1-13}
 & \multicolumn{4}{c|}{outer iterations} &
 \multicolumn{4}{c|}{local Hessian assemblies/solves} & \multicolumn{4}{c}{parallel patch inner solves} \\
\cline{2-13}
\diagbox{$h$}{$\epsilon$} & 
$10^0$ & $10^{-1}$ & $10^{-2}$ & $10^{-3}$ &
$10^0$ & $10^{-1}$ & $10^{-2}$ & $10^{-3}$ & 
$10^0$ & $10^{-1}$ & $10^{-2}$ & $10^{-3}$ \\
\cline{1-13}
$2^{-6}$ & 
9 & 15 & 17 & 17 &    
1.4E5 & 2.1E5 & 2.0E5 & 1.9E5 &
69 & 123 & 135 & 132 \\
$2^{-7}$ &
7 & 13 & 13 & 12 &    
4.4E5 & 6.1E5 & 5.7E5 & 5.2E5 &
56 & 89 & 90 & 81 \\
$2^{-8}$ & 
9 & 15 & 17 & 18 &   
1.9E6 & 2.6E6 & 2.7E6 & 2.7E6 &
77 & 131 & 141 & 145 \\
$2^{-9}$ & 
8 & 13 & 14 & 14 &   
6.1E6 & 7.5E6 & 7.7E6 & 7.5E6 &
66 & 95 & 100 & 96 \\
\cline{1-13}
& \multicolumn{4}{c|}{global reductions} & \multicolumn{4}{c|}{global energy evaluations} & \multicolumn{4}{c}{global gradient evaluations} \\
\cline{2-13}
\diagbox{$h$}{$\epsilon$} & 
$10^0$ & $10^{-1}$ & $10^{-2}$ & $10^{-3}$ &
$10^0$ & $10^{-1}$ & $10^{-2}$ & $10^{-3}$ & 
$10^0$ & $10^{-1}$ & $10^{-2}$ & $10^{-3}$ \\
\cline{1-13}
$2^{-6}$ & 
9 & 15 & 17 & 17 &   
18 & 30 & 35 & 34 &  
9 & 15 & 17 & 17 \\
$2^{-7}$ & 
7 & 13 & 13 & 12 &   
16 & 27 & 27 & 24 &   
7 & 13 & 13 & 12 \\
$2^{-8}$ & 
9 & 15 & 17 & 18 &   
18 & 30 & 34 & 36 &  
9 & 15 & 17 & 18 \\
$2^{-9}$ & 
8 & 13 & 14 & 14 &    
17 & 26 & 28 & 28 &  
8 & 13 & 14 & 14 \\
\cline{1-13}
\multicolumn{13}{c}{} \\
\multicolumn{13}{c}{$\beta = 20$} \\
\cline{1-13}
 & \multicolumn{4}{c|}{outer iterations} &
 \multicolumn{4}{c|}{local Hessian assemblies/solves} & \multicolumn{4}{c}{parallel patch inner solves} \\
\cline{2-13}
\diagbox{$h$}{$\epsilon$} & 
$10^0$ & $10^{-1}$ & $10^{-2}$ & $10^{-3}$ &
$10^0$ & $10^{-1}$ & $10^{-2}$ & $10^{-3}$ & 
$10^0$ & $10^{-1}$ & $10^{-2}$ & $10^{-3}$ \\
\cline{1-13}
$2^{-6}$ & 
7 & 13 & 17 & 17 &   
1.0E5 & 1.9E5 & 2.1E5 & 2.0E5 &
54 & 109 & 135 & 137 \\
$2^{-7}$ &  
6 & 12 & 13 & 12 &   
3.7E5 & 5.6E5 & 5.7E5 & 5.1E5 &
52 & 86 & 92 & 81 \\
$2^{-8}$ & 
8 & 15 & 16 & 17 &   
1.6E6 & 2.5E6 & 2.5E6 & 2.5E6 &
62 & 130 & 135 & 142 \\
$2^{-9}$ & 
7 & 12 & 14 & 14 &   
5.9E6 & 7.3E6 & 7.7E6 & 7.5E6 &
60 & 95 & 105 & 102 \\
\cline{1-13}
& \multicolumn{4}{c|}{global reductions} & \multicolumn{4}{c|}{global energy evaluations} & \multicolumn{4}{c}{global gradient evaluations} \\
\cline{2-13}
\diagbox{$h$}{$\epsilon$} & 
$10^0$ & $10^{-1}$ & $10^{-2}$ & $10^{-3}$ &
$10^0$ & $10^{-1}$ & $10^{-2}$ & $10^{-3}$ & 
$10^0$ & $10^{-1}$ & $10^{-2}$ & $10^{-3}$ \\
\cline{1-13}
$2^{-6}$ & 
7 & 13 & 17 & 17 &  
14 & 27 & 34 & 34 &  
7 & 13 & 17 & 17 \\
$2^{-7}$ & 
6 & 12 & 13 & 12 &    
12 & 24 & 28 & 24 &  
6 & 12 & 13 & 12 \\
$2^{-8}$ & 
8 & 15 & 16 & 17 &   
16 & 30 & 34 & 35 &  
8 & 15 & 16 & 17 \\
$2^{-9}$ & 
7 & 12 & 14 & 14 &   
15 & 26 & 28 & 29 &  
7 & 12 & 14 & 14 \\
\cline{1-13}
\multicolumn{13}{c}{} \\
\multicolumn{13}{c}{$\beta = 30$} \\
\cline{1-13}
 & \multicolumn{4}{c|}{outer iterations} &
 \multicolumn{4}{c|}{local Hessian assemblies/solves} & \multicolumn{4}{c}{parallel patch inner solves} \\
\cline{2-13}
\diagbox{$h$}{$\epsilon$} & 
$10^0$ & $10^{-1}$ & $10^{-2}$ & $10^{-3}$ &
$10^0$ & $10^{-1}$ & $10^{-2}$ & $10^{-3}$ & 
$10^0$ & $10^{-1}$ & $10^{-2}$ & $10^{-3}$ \\
\cline{1-13}
$2^{-6}$ & 
6 & 12 & 16 & 17 &  
8.8E4 & 1.8E5 & 2.1E5 & 2.0E5 &
49 & 98 & 130 & 137 \\
$2^{-7}$ &  
6 & 10 & 13 & 11 &    
3.7E5 & 5.1E5 & 5.7E5 & 4.8E5 &
52 & 76 & 89 & 77 \\
$2^{-8}$ & 
6 & 12 & 16 & 17 &    
1.3E6 & 2.2E6 & 2.5E6 & 2.5E6 &
54 & 104 & 138 & 143 \\
$2^{-9}$ & 
6 & 10 & 13 & 13 &    
5.1E6 & 6.1E6 & 7.0E6 & 6.8E6 &
55 & 82 & 94 & 96 \\
\cline{1-13}
& \multicolumn{4}{c|}{global reductions} & \multicolumn{4}{c|}{global energy evaluations} & \multicolumn{4}{c}{global gradient evaluations} \\
\cline{2-13}
\diagbox{$h$}{$\epsilon$} & 
$10^0$ & $10^{-1}$ & $10^{-2}$ & $10^{-3}$ &
$10^0$ & $10^{-1}$ & $10^{-2}$ & $10^{-3}$ & 
$10^0$ & $10^{-1}$ & $10^{-2}$ & $10^{-3}$ \\
\cline{1-13}
$2^{-6}$ & 
6 & 12 & 16 & 17 &   
13 & 25 & 33 & 35 &  
6 & 12 & 16 & 17 \\
$2^{-7}$ &  
6 & 10 & 13 & 11 &   
12 & 20 & 26 & 22 &  
6 & 10 & 13 & 11 \\
$2^{-8}$ & 
6 & 12 & 16 & 17 &   
12 & 24 & 32 & 36 &   
6 & 12 & 16 & 17 \\
$2^{-9}$ & 
6 & 10 & 13 & 13 &   
13 & 20 & 27 & 28 &  
6 & 10 & 13 & 13 \\
\cline{1-13}
\end{tabular}
}
\end{table}

\begin{table}
\label{Table:dependence_Ex2_coarsening}
\centering
\caption{Numerical results of the proposed multilevel method (\cref{Alg:multilevel_FAS}, coarsening factor $\gamma = 2^{-2}$) satisfying the stopping criterion~\eqref{stop_multilevel}, for various values of the augmented Lagrangian parameter $\epsilon$ and mesh size $h$ (\cref{Ex:Ex2}).}
\resizebox{0.99\textwidth}{!}{
\begin{tabular}{c|cccc|cccc|cccc}
\multicolumn{13}{c}{$\beta = 10$} \\
\cline{1-13}
 & \multicolumn{4}{c|}{outer iterations} &
 \multicolumn{4}{c|}{local Hessian assemblies/solves} & \multicolumn{4}{c}{parallel patch inner solves} \\
\cline{2-13}
\diagbox{$h$}{$\epsilon$} & 
$10^0$ & $10^{-1}$ & $10^{-2}$ & $10^{-3}$ &
$10^0$ & $10^{-1}$ & $10^{-2}$ & $10^{-3}$ & 
$10^0$ & $10^{-1}$ & $10^{-2}$ & $10^{-3}$ \\
\cline{1-13}
$2^{-6}$ & 
9 & 16 & 19 & 20 &  
1.4E5 & 2.2E5 & 2.6E5 & 2.9E5 &
87 & 161 & 192 & 224 \\
$2^{-7}$ &
8 & 11 & 15 & 19 &   
4.2E5 & 5.4E5 & 7.8E5 & 1.0E6 &
79 & 117 & 171 & 236 \\
$2^{-8}$ & 
10 & 17 & 20 & 21 &   
1.9E6 & 2.9E6 & 3.6E6 & 4.1E6 &
103 & 177 & 218 & 243 \\
$2^{-9}$ &  
8 & 12 & 15 & 19 &   
4.9E6 & 7.2E6 & 9.9E6 & 1.4E7 &
86 & 134 & 187 & 257 \\
\cline{1-13}
& \multicolumn{4}{c|}{global reductions} & \multicolumn{4}{c|}{global energy evaluations} & \multicolumn{4}{c}{global gradient evaluations} \\
\cline{2-13}
\diagbox{$h$}{$\epsilon$} & 
$10^0$ & $10^{-1}$ & $10^{-2}$ & $10^{-3}$ &
$10^0$ & $10^{-1}$ & $10^{-2}$ & $10^{-3}$ & 
$10^0$ & $10^{-1}$ & $10^{-2}$ & $10^{-3}$ \\
\cline{1-13}
$2^{-6}$ & 
9 & 16 & 19 & 20 &  
20 & 32 & 38 & 42 &  
9 & 16 & 19 & 20 \\
$2^{-7}$ & 
8 & 11 & 15 & 19 &    
16 & 24 & 30 & 38 &  
8 & 11 & 15 & 19 \\
$2^{-8}$ & 
10 & 17 & 20 & 21 &    
20 & 34 & 40 & 43 &  
10 & 17 & 20 & 21 \\
$2^{-9}$ & 
8 & 12 & 15 & 19 &   
18 & 24 & 31 & 39 &  
8 & 12 & 15 & 19 \\
\cline{1-13}
\multicolumn{13}{c}{} \\
\multicolumn{13}{c}{$\beta = 20$} \\
\cline{1-13}
 & \multicolumn{4}{c|}{outer iterations} &
 \multicolumn{4}{c|}{local Hessian assemblies/solves} & \multicolumn{4}{c}{parallel patch inner solves} \\
\cline{2-13}
\diagbox{$h$}{$\epsilon$} & 
$10^0$ & $10^{-1}$ & $10^{-2}$ & $10^{-3}$ &
$10^0$ & $10^{-1}$ & $10^{-2}$ & $10^{-3}$ & 
$10^0$ & $10^{-1}$ & $10^{-2}$ & $10^{-3}$ \\
\cline{1-13}
$2^{-6}$ & 
7 & 14 & 19 & 19 &   
1.2E5 & 2.0E5 & 2.6E5 & 2.8E5 &
74 & 141 & 195 & 210 \\
$2^{-7}$ &  
7 & 10 & 14 & 18 &   
4.0E5 & 5.0E5 & 7.1E5 & 9.8E5 &
68 & 102 & 162 & 224 \\
$2^{-8}$ & 
8 & 14 & 19 & 21 &   
1.6E6 & 2.5E6 & 3.5E6 & 4.0E6 &
85 & 155 & 207 & 234 \\
$2^{-9}$ & 
7 & 11 & 15 & 18 &   
4.5E6 & 6.4E6 & 9.6E6 & 1.3E7 &
77 & 124 & 186 & 245 \\
\cline{1-13}
& \multicolumn{4}{c|}{global reductions} & \multicolumn{4}{c|}{global energy evaluations} & \multicolumn{4}{c}{global gradient evaluations} \\
\cline{2-13}
\diagbox{$h$}{$\epsilon$} & 
$10^0$ & $10^{-1}$ & $10^{-2}$ & $10^{-3}$ &
$10^0$ & $10^{-1}$ & $10^{-2}$ & $10^{-3}$ & 
$10^0$ & $10^{-1}$ & $10^{-2}$ & $10^{-3}$ \\
\cline{1-13}
$2^{-6}$ & 
7 & 14 & 19 & 19 &  
15 & 30 & 38 & 38 &  
7 & 14 & 19 & 19 \\
$2^{-7}$ & 
7 & 10 & 14 & 18 &    
14 & 20 & 29 & 36 &  
7 & 10 & 14 & 18 \\
$2^{-8}$ & 
8 & 14 & 19 & 21 &    
17 & 28 & 38 & 42 &  
8 & 14 & 19 & 21 \\
$2^{-9}$ & 
7 & 11 & 15 & 18 &   
14 & 23 & 30 & 38 &  
7 & 11 & 15 & 18 \\
\cline{1-13}
\multicolumn{13}{c}{} \\
\multicolumn{13}{c}{$\beta = 30$} \\
\cline{1-13}
 & \multicolumn{4}{c|}{outer iterations} &
 \multicolumn{4}{c|}{local Hessian assemblies/solves} & \multicolumn{4}{c}{parallel patch inner solves} \\
\cline{2-13}
\diagbox{$h$}{$\epsilon$} & 
$10^0$ & $10^{-1}$ & $10^{-2}$ & $10^{-3}$ &
$10^0$ & $10^{-1}$ & $10^{-2}$ & $10^{-3}$ & 
$10^0$ & $10^{-1}$ & $10^{-2}$ & $10^{-3}$ \\
\cline{1-13}
$2^{-6}$ & 
6 & 13 & 17 & 20 &   
1.0E5 & 1.9E5 & 2.4E5 & 2.8E5 &
58 & 126 & 179 & 214 \\
$2^{-7}$ & 
6 & 10 & 13 & 17 &    
3.7E5 & 5.0E5 & 6.6E5 & 9.3E5 &
61 & 106 & 149 & 211 \\
$2^{-8}$ & 
7 & 14 & 18 & 21 &    
1.4E6 & 2.5E6 & 3.3E6 & 4.0E6 &
76 & 153 & 200 & 237 \\
$2^{-9}$ & 
6 & 10 & 14 & 17 &   
4.1E6 & 5.9E6 & 9.0E6 & 1.2E7 &
67 & 112 & 176 & 232 \\
\cline{1-13}
& \multicolumn{4}{c|}{global reductions} & \multicolumn{4}{c|}{global energy evaluations} & \multicolumn{4}{c}{global gradient evaluations} \\
\cline{2-13}
\diagbox{$h$}{$\epsilon$} & 
$10^0$ & $10^{-1}$ & $10^{-2}$ & $10^{-3}$ &
$10^0$ & $10^{-1}$ & $10^{-2}$ & $10^{-3}$ & 
$10^0$ & $10^{-1}$ & $10^{-2}$ & $10^{-3}$ \\
\cline{1-13}
$2^{-6}$ & 
6 & 13 & 17 & 20 &    
12 & 26 & 34 & 41 & 
6 & 13 & 17 & 20 \\
$2^{-7}$ & 
6 & 10 & 13 & 17 &   
13 & 21 & 27 & 35 &  
6 & 10 & 13 & 17 \\
$2^{-8}$ & 
7 & 14 & 18 & 21 &     
15 & 28 & 36 & 42 &   
7 & 14 & 18 & 21 \\
$2^{-9}$ & 
6 & 10 & 14  & 17 &   
12 & 22 & 29 & 34 &  
6 & 10 & 14  & 17 \\
\cline{1-13}
\end{tabular}
}
\end{table}

To assess the robustness and computational efficiency of \cref{Alg:multilevel_FAS} with respect to the augmented Lagrangian parameter $\epsilon$ and the mesh size $h$, we report in \cref{Table:dependence_Ex1,Table:dependence_Ex2} numerical results, including idealized parallel cost proxies, for various values of $\epsilon$ and $h$ for \cref{Ex:Ex1,Ex:Ex2}, respectively.
The following stopping criterion is used:
\begin{equation}
\label{stop_multilevel}
\left| \frac{F^{\epsilon} (\bu^{(n)}) - \min F^{\epsilon} }{  \min F^{\epsilon} } \right| < 10^{-3}.
\end{equation}

More precisely, we report:
\begin{itemize}
    \item Outer iterations: number of outer iterations to satisfy~\eqref{stop_multilevel}.
    \item Local Hessian assemblies/solves: number of assemblies and solves of local patch-based Hessian systems (equal for \cref{Alg:multilevel_FAS}).
    \item Parallel patch inner solves: ideal parallel cost for solving all local Hessian systems across levels and patches.
    \item Global reductions: evaluation of $\sum_{j=1}^J \sum_{k=1}^{n_j} \bw_{j,k}$ for $\bw_{j,k} \in V_{j,k}$.
    \item Global energy evaluations: evaluation of $F^{\epsilon}(\bv)$ for $\bv \in V$ (for backtracking).
    \item Global gradient evaluations: evaluation of $(F^{\epsilon})'(\bv)$ for $\bv \in V$.
\end{itemize}
For the proposed multilevel methods, the numbers of outer iterations, global reductions, and gradient evaluations coincide, as each is performed once per outer iteration, highlighting communication efficiency.
Since the number of Newton iterations varies across local problems, the parallel patch inner solves are defined as the sum of the maximum numbers of inner Newton iterations over all parallel phases, reflecting an idealized per-process cost under full parallelization, while the total number of local Hessian assemblies/solves measures the overall computational cost.

We observe that the number of outer iterations does not increase significantly as $\epsilon$ decreases, which is consistent with the $\epsilon$-robust behavior predicted by \cref{Thm:convergence} for the exact local-solver variant.
The observed weak dependence on $h$ provides numerical evidence of mesh robustness.

In terms of computational cost, we observe that the number of local Hessian assemblies/solves increases by approximately a factor of four when the mesh size $h$ is halved.
This indicates that the total computational cost of \cref{Alg:multilevel_FAS} scales nearly linearly with the number of degrees of freedom.
In contrast, the number of parallel patch inner solves increases only very slowly as the mesh is refined.
Thus, under the idealized parallel cost model used here, the parallel patch inner-solve count remains nearly constant even for large-scale problems, highlighting the potential of the proposed multilevel method for parallel implementation.

In terms of communication-related counts, we observe that the numbers of global reductions and global gradient evaluations coincide with the number of outer iterations, indicating a low communication-related count per iteration.
Moreover, the number of global energy evaluations is only about two to three times the number of outer iterations.
This indicates that the backtracking line search requires only a small number of additional energy evaluations per iteration, incurring modest communication-related overhead in the idealized count model while significantly improving the convergence rate, as observed in \cref{Fig:backt_Ex1,Fig:backt_Ex2}.

Finally, to examine the effect of the coarsening factor in the multilevel structure, we present in \cref{Table:dependence_Ex1_coarsening,Table:dependence_Ex2_coarsening} numerical results for \cref{Alg:multilevel_FAS} with coarsening factor $\gamma = 2^{-2}$ in the multilevel space decomposition~\eqref{space_decomposition}.
We observe that the number of outer iterations changes only mildly as the mesh is refined, while the robustness with respect to the augmented Lagrangian parameter $\epsilon$ deteriorates compared to the case $\gamma = 2^{-1}$, especially for \cref{Ex:Ex2}; the number of outer iterations increases slightly as $\epsilon$ decreases.

Compared with the case $\gamma = 2^{-1}$ in \cref{Table:dependence_Ex1,Table:dependence_Ex2}, the number of outer iterations increases.
Thus, although fewer parallel subspaces are used, the number of local Hessian assemblies/solves, which measures total computational cost, does not improve and may even worsen in some cases.
Accordingly, the number of parallel patch inner solves also increases.
This indicates that using a smaller coarsening factor does not improve either the total-cost metric or the idealized parallel-cost proxy.

\subsection{Comparison with multilevel Newton--Krylov--Schwarz}

\begin{table}
\label{Table:comparison_Ex1}
\centering
\caption{Numerical results of the proposed multilevel method (\cref{Alg:multilevel_FAS}, Proposed) and the parallel Newton--Krylov--Schwarz method (NKS) satisfying the stopping criterion~\eqref{stop_multilevel}, for various mesh sizes $h$ (\cref{Ex:Ex1}, $\epsilon = 10^{-2}$).}
\resizebox{0.99\textwidth}{!}{
\begin{tabular}{c|cc|cc|cc|cc}
\multicolumn{9}{c}{$\beta = 10$} \\
\cline{1-9}
 & \multicolumn{2}{c|}{outer iterations} &
 \multicolumn{2}{c|}{global Hessian assemblies} &
 \multicolumn{2}{c|}{local Hessian assemblies} & \multicolumn{2}{c}{local Hessian solves} \\
\cline{2-9}
$h$ & 
Proposed & NKS & Proposed & NKS &
Proposed & NKS & Proposed & NKS \\
\cline{1-9}
$2^{-6}$ & 9 & 7 & 0 & 7 & 1.5E5 & 4.0E4 & 1.5E5 & 1.3E6 \\
$2^{-7}$ & 9 & 7 & 0 & 7 & 5.7E5 & 1.6E5 & 5.7E5 & 5.5E6 \\
$2^{-8}$ & 10 & 7 & 0 & 7 & 2.2E6 & 6.2E5 & 2.2E6 & 2.3E7 \\
$2^{-9}$ & 10 & 7 & 0 & 7 & 8.0E6 & 2.5E6 & 8.0E6 & 9.7E7 \\
\cline{1-9}
& \multicolumn{2}{c|}{parallel patch inner solves} & \multicolumn{2}{c|}{global reductions} &
\multicolumn{2}{c|}{global energy evaluations} & \multicolumn{2}{c}{global gradient evaluations} \\
\cline{2-9}
$h$ & 
Proposed & NKS & Proposed & NKS &
Proposed & NKS & Proposed & NKS \\
\cline{1-9}
$2^{-6}$ & 69 & 229 & 9 & 229 & 20 & 15 & 9 & 7 \\
$2^{-7}$ & 72 & 244 & 9 & 244 & 19 & 15 & 9 & 7 \\
$2^{-8}$ & 83 & 261 & 10 & 261 & 21 & 15 & 10 & 7 \\
$2^{-9}$ & 93 & 276 & 10 & 276 & 22 & 15 & 10 & 7 \\
\cline{1-9}
\multicolumn{9}{c}{} \\
\multicolumn{9}{c}{$\beta = 20$} \\
\cline{1-9}
 & \multicolumn{2}{c|}{outer iterations} &
 \multicolumn{2}{c|}{global Hessian assemblies} &
 \multicolumn{2}{c|}{local Hessian assemblies} & \multicolumn{2}{c}{local Hessian solves} \\
\cline{2-9}
$h$ & 
Proposed & NKS & Proposed & NKS &
Proposed & NKS & Proposed & NKS \\
\cline{1-9}
$2^{-6}$ & 8 & 8 & 0 & 8 & 1.4E5 & 4.5E4 & 1.4E5 & 1.5E6 \\
$2^{-7}$ & 9 & 8 & 0 & 8 & 5.7E5 & 1.8E5 & 5.7E5 & 6.2E6 \\
$2^{-8}$ & 10 & 8 & 0 & 8 & 2.1E6 & 7.1E5 & 2.1E6 & 2.6E7 \\
$2^{-9}$ & 10 & 8 & 0 & 8 & 7.9E6 & 2.8E6 & 7.9E6 & 1.1E8 \\
\cline{1-9}
& \multicolumn{2}{c|}{parallel patch inner solves} & \multicolumn{2}{c|}{global reductions} &
\multicolumn{2}{c|}{global energy evaluations} & \multicolumn{2}{c}{global gradient evaluations} \\
\cline{2-9}
$h$ & 
Proposed & NKS & Proposed & NKS &
Proposed & NKS & Proposed & NKS \\
\cline{1-9}
$2^{-6}$ & 65 & 257 & 8 & 257 & 17 & 17 & 8 & 8 \\
$2^{-7}$ & 76 & 279 & 9 & 279 & 19 & 17 & 9 & 8 \\
$2^{-8}$ & 87 & 295 & 10 & 295 & 22 & 17 & 10 & 8 \\
$2^{-9}$ & 92 & 310 & 10 & 310 & 21 & 17 & 10 & 8 \\
\cline{1-9}
\multicolumn{9}{c}{} \\
\multicolumn{9}{c}{$\beta = 30$} \\
\cline{1-9}
 & \multicolumn{2}{c|}{outer iterations} &
 \multicolumn{2}{c|}{global Hessian assemblies} &
 \multicolumn{2}{c|}{local Hessian assemblies} & \multicolumn{2}{c}{local Hessian solves} \\
\cline{2-9}
$h$ & 
Proposed & NKS & Proposed & NKS &
Proposed & NKS & Proposed & NKS \\
\cline{1-9}
$2^{-6}$ & 8 & 8 & 0 & 8 & 1.5E5 & 4.5E4 & 1.5E5 & 1.4E6 \\
$2^{-7}$ & 9 & 8 & 0 & 8 & 5.6E5 & 1.8E5 & 5.6E5 & 5.9E6 \\
$2^{-8}$ & 9 & 8 & 0 & 8 & 2.0E6 & 7.1E5 & 2.0E6 & 2.5E7 \\
$2^{-9}$ & 10 & 8 & 0 & 8 & 7.7E6 & 2.8E6 & 7.7E6 & 1.0E8 \\
\cline{1-9}
& \multicolumn{2}{c|}{parallel patch inner solves} & \multicolumn{2}{c|}{global reductions} &
\multicolumn{2}{c|}{global energy evaluations} & \multicolumn{2}{c}{global gradient evaluations} \\
\cline{2-9}
$h$ & 
Proposed & NKS & Proposed & NKS &
Proposed & NKS & Proposed & NKS \\
\cline{1-9}
$2^{-6}$ & 66 & 248 & 8 & 248 & 17 & 16 & 8 & 8 \\
$2^{-7}$ & 75 & 265 & 9 & 265 & 20 & 16 & 9 & 8 \\
$2^{-8}$ & 81 & 283 & 9 & 283 & 19 & 16 & 9 & 8 \\
$2^{-9}$ & 94 & 296 & 10 & 296 & 22 & 16 & 10 & 8 \\
\cline{1-9}
\end{tabular}
}
\end{table}

\begin{table}
\label{Table:comparison_Ex2}
\centering
\caption{Numerical results of the proposed multilevel method (\cref{Alg:multilevel_FAS}, Proposed) and the parallel Newton--Krylov--Schwarz method (NKS) satisfying the stopping criterion~\eqref{stop_multilevel}, for various mesh sizes $h$ (\cref{Ex:Ex2}, $\epsilon = 10^{-2}$).}
\resizebox{0.99\textwidth}{!}{
\begin{tabular}{c|cc|cc|cc|cc}
\multicolumn{9}{c}{$\beta = 10$} \\
\cline{1-9}
 & \multicolumn{2}{c|}{outer iterations} &
 \multicolumn{2}{c|}{global Hessian assemblies} &
 \multicolumn{2}{c|}{local Hessian assemblies} & \multicolumn{2}{c}{local Hessian solves} \\
\cline{2-9}
$h$ & 
Proposed & NKS & Proposed & NKS &
Proposed & NKS & Proposed & NKS \\
\cline{1-9}
$2^{-6}$ & 11 & 6 & 0 & 6 & 1.8E5 & 3.4E4 & 1.8E5 & 8.6E5 \\
$2^{-7}$ & 12 & 6 & 0 & 6 & 7.0E5 & 1.3E5 & 7.0E5 & 3.6E6 \\
$2^{-8}$ & 13 & 6 & 0 & 6 & 2.7E6 & 5.3E5 & 2.7E6 & 1.5E7 \\
$2^{-9}$ & 13 & 6 & 0 & 6 & 9.5E6 & 2.1E6 & 9.5E6 & 6.2E7 \\
\cline{1-9}
& \multicolumn{2}{c|}{parallel patch inner solves} & \multicolumn{2}{c|}{global reductions} &
\multicolumn{2}{c|}{global energy evaluations} & \multicolumn{2}{c}{global gradient evaluations} \\
\cline{2-9}
$h$ & 
Proposed & NKS & Proposed & NKS &
Proposed & NKS & Proposed & NKS \\
\cline{1-9}
$2^{-6}$ & 117 & 150 & 11 & 150 & 24 & 12 & 11 & 6 \\
$2^{-7}$ & 139 & 159 & 12 & 159 & 25 & 12 & 12 & 6 \\
$2^{-8}$ & 154 & 168 & 13 & 168 & 27 & 12 & 12 & 6 \\
$2^{-9}$ & 153 & 176 & 13 & 176 & 27 & 12 & 13 & 6 \\
\cline{1-9}
\multicolumn{9}{c}{} \\
\multicolumn{9}{c}{$\beta = 20$} \\
\cline{1-9}
 & \multicolumn{2}{c|}{outer iterations} &
 \multicolumn{2}{c|}{global Hessian assemblies} &
 \multicolumn{2}{c|}{local Hessian assemblies} & \multicolumn{2}{c}{local Hessian solves} \\
\cline{2-9}
$h$ & 
Proposed & NKS & Proposed & NKS &
Proposed & NKS & Proposed & NKS \\
\cline{1-9}
$2^{-6}$ & 10 & 6 & 0 & 6 & 1.7E5 & 3.4E4 & 1.7E5 & 8.2E5 \\
$2^{-7}$ & 11 & 6 & 0 & 6 & 6.7E5 & 1.3E5 & 6.7E5 & 3.4E6 \\
$2^{-8}$ & 12 & 6 & 0 & 6 & 2.5E6 & 5.3E5 & 2.5E6 & 1.4E7 \\
$2^{-9}$ & 13 & 6 & 0 & 6 & 9.2E6 & 2.1E6 & 9.2E6 & 5.9E7 \\
\cline{1-9}
& \multicolumn{2}{c|}{parallel patch inner solves} & \multicolumn{2}{c|}{global reductions} &
\multicolumn{2}{c|}{global energy evaluations} & \multicolumn{2}{c}{global gradient evaluations} \\
\cline{2-9}
$h$ & 
Proposed & NKS & Proposed & NKS &
Proposed & NKS & Proposed & NKS \\
\cline{1-9}
$2^{-6}$ & 112 & 144 & 10 & 144 & 21 & 12 & 10 & 6 \\
$2^{-7}$ & 129 & 153 & 11 & 153 & 24 & 12 & 11 & 6 \\
$2^{-8}$ & 158 & 160 & 12 & 160 & 25 & 12 & 12 & 6 \\
$2^{-9}$ & 158 & 168 & 13 & 168 & 27 & 12 & 13 & 6 \\
\cline{1-9}
\multicolumn{9}{c}{} \\
\multicolumn{9}{c}{$\beta = 30$} \\
\cline{1-9}
 & \multicolumn{2}{c|}{outer iterations} &
 \multicolumn{2}{c|}{global Hessian assemblies} &
 \multicolumn{2}{c|}{local Hessian assemblies} & \multicolumn{2}{c}{local Hessian solves} \\
\cline{2-9}
$h$ & 
Proposed & NKS & Proposed & NKS &
Proposed & NKS & Proposed & NKS \\
\cline{1-9}
$2^{-6}$ & 10 & 11 & 0 & 11 & 1.7E5 & 6.3E4 & 1.7E5 & 1.5E6 \\
$2^{-7}$ & 11 & 11 & 0 & 11 & 6.6E5 & 2.5E5 & 6.6E5 & 6.2E6 \\
$2^{-8}$ & 12 & 11 & 0 & 11 & 2.5E6 & 9.7E5 & 2.5E6 & 2.6E7 \\
$2^{-9}$ & 12 & 11 & 0 & 11 & 8.6E6 & 3.8E6 & 8.6E6 & 1.1E8 \\
\cline{1-9}
& \multicolumn{2}{c|}{parallel patch inner solves} & \multicolumn{2}{c|}{global reductions} &
\multicolumn{2}{c|}{global energy evaluations} & \multicolumn{2}{c}{global gradient evaluations} \\
\cline{2-9}
$h$ & 
Proposed & NKS & Proposed & NKS &
Proposed & NKS & Proposed & NKS \\
\cline{1-9}
$2^{-6}$ & 110 & 259 & 10 & 259 & 21 & 23 & 10 & 11 \\
$2^{-7}$ & 130 & 278 & 11 & 278 & 24 & 23 & 11 & 11 \\
$2^{-8}$ & 144 & 292 & 12 & 292 & 25 & 23 & 12 & 11 \\
$2^{-9}$ & 145 & 305 & 12 & 305 & 26 & 23 & 12 & 11 \\
\cline{1-9}
\end{tabular}
}
\end{table}

To assess the nonlinear multilevel decomposition used in the proposed methods, we compare \cref{Alg:multilevel_FAS} with the multilevel Newton--Krylov--Schwarz~\cite{CGKMY:1998,CGKT:1994} method for solving the nearly semicoercive formulation~\eqref{mixed_FEM_nearly}, using the same multilevel space decomposition.
In the multilevel Newton--Krylov--Schwarz method, the nonlinearity is handled by Newton iterations, and each linearized subproblem involving the global Hessian is solved by a parallel multilevel method.
Here, we use the damped Newton method~\cite{BV:2004} to ensure global convergence, and each linear subproblem is solved by the preconditioned conjugate gradient method with a standard multilevel additive Schwarz preconditioner~\cite{BPX:1990,DSW:1996,Zhang:1992,Zhang:1994}.
For the outer Newton iteration, we use the stopping criterion~\eqref{stop_multilevel}, as in \cref{Alg:multilevel_FAS}, while for the inner preconditioned conjugate gradient iterations we require the relative residual to be less than $10^{-5}$.

In \cref{Table:comparison_Ex1,Table:comparison_Ex2}, we report numerical results for \cref{Alg:multilevel_FAS} and the multilevel Newton--Krylov--Schwarz method applied to \cref{Ex:Ex1,Ex:Ex2}, respectively.
For both algorithms, the numbers of outer iterations, as well as the numbers of global energy and gradient evaluations, are comparable.
While \cref{Alg:multilevel_FAS} does not require assembly of global Hessians, the multilevel Newton--Krylov--Schwarz method requires one global Hessian assembly per outer iteration.
On the other hand, \cref{Alg:multilevel_FAS} requires about four times more local Hessian assemblies, but each local problem is solved only once, whereas in Newton--Krylov--Schwarz it must be solved multiple times due to the preconditioned conjugate gradient iterations.
As a result, the number of local Hessian solves in Newton--Krylov--Schwarz is about ten times larger than in \cref{Alg:multilevel_FAS}, indicating a significantly higher total computational cost.

In terms of the number of parallel patch inner solves, which serves as an implementation-independent proxy for ideal parallel cost, Newton--Krylov--Schwarz requires three to four times more than \cref{Alg:multilevel_FAS}.
Thus, under the idealized cost model used here, it incurs a proportionally larger predicted parallel cost.
Finally, \cref{Alg:multilevel_FAS} requires only one global reduction per outer iteration, whereas Newton--Krylov--Schwarz requires two reductions per inner preconditioned conjugate gradient iteration, leading to substantially higher communication-related counts.
Overall, \cref{Alg:multilevel_FAS} shows lower total cost proxies and lower communication-related counts in these experiments.

\begin{remark}
\label{Rem:Newton}
In \cref{Table:comparison_Ex1,Table:comparison_Ex2}, we observe that the number of outer iterations of the Newton--Krylov--Schwarz method remains unchanged as the mesh size $h$ is refined in all cases.
This phenomenon is consistent with the well-known mesh-independence principle of the Newton method; see, e.g.,~\cite{ABPR:1986}.
\end{remark}

\section{Conclusion}
\label{Sec:Conclusion}
In this paper, we proposed multilevel methods for solving the Darcy--Forchheimer model based on a nearly semicoercive formulation.
We first proved that the Darcy--Forchheimer model, formulated as a mixed problem, can be reduced to a nearly semicoercive convex optimization problem via the augmented Lagrangian method.
Then, based on this nearly semicoercive formulation, we developed a multilevel method utilizing multilevel patch-based space decomposition, with convergence analysis demonstrating robustness with respect to the augmented Lagrangian parameter $\epsilon$.
Additionally, we introduced a backtracking line search and an FAS approach for reducing the cost of the local subproblems within the proposed multilevel method.

One interesting direction for future research is the rigorous mathematical analysis of the convergence behavior of the proposed multilevel methods with respect to the mesh size $h$.
While $h$-independent convergence rates of parallel multilevel methods for linear problems have been extensively studied in the literature~(e.g.,~\cite{Zhang:1992,Zhang:1994}), extending these results to our nonlinear setting poses significant challenges.
One major difficulty lies in controlling the $L^3(\Omega)^d$-norm in addition to the $H(\div; \Omega)$-norm, as discussed in \cref{Sec:Multilevel}.

\appendix
\section{Convergence analysis of the augmented Lagrangian method}
\label{App:Aug}
In this appendix, we present a proof of \cref{Thm:aug} in an abstract setting.
That is, we provide a convergence rate analysis of the augmented Lagrangian method~\cite{Hestenes:1969,Powell:1969} for solving convex optimization problems with linear constraints.

Let $V$ and $W$ be finite-dimensional real vector spaces equipped with inner products $(\cdot, \cdot)$ and associated norms $\| \cdot \|$.
We consider the following general convex optimization problem with a linear constraint:
\begin{equation}
\label{general}
\min_{v \in V} F(v) 
\quad \text{subject to} \quad Bv = g,
\end{equation}
where $B \colon V \rightarrow W$ is a linear operator, $F \colon V \rightarrow \overline{\mathbb{R}}$ is a convex functional, and $g \in W$.
Let $u \in V$ denote a solution to~\eqref{general}. The augmented Lagrangian method for solving~\eqref{general} is summarized in \cref{Alg:aug_general}.

\begin{algorithm}
\caption{Augmented Lagrangian method for solving~\eqref{general}}
\begin{algorithmic}[]
\label{Alg:aug_general}
\STATE Given $\epsilon > 0$:
\STATE Choose $p^{(0)} \in W$.
\FOR{$n=0,1,2,\dots$}
    \STATE $\displaystyle
    u^{(n+1)} = \operatornamewithlimits{\arg\min}_{v \in V} \left\{ F (v) + ( p^{(n)}, Bv - g ) + \frac{1}{2\epsilon} \| Bv - g \|^2 \right\}
    $
    \STATE $\displaystyle
    p^{(n+1)} = p^{(n)} + \epsilon^{-1} ( B u^{(n+1)} - g)
    $
\ENDFOR
\end{algorithmic}
\end{algorithm}

We observe that \cref{Alg:aug_general} reduces to \cref{Alg:aug} if we set
\begin{equation}
\label{aug_general_reduction}
V \leftarrow X_h, \quad
W \leftarrow M_h, \quad
\| \cdot \| \leftarrow \| \cdot \|_{L^2 (\Omega)}, \quad
F(v) \leftarrow F (\bv), \quad
B \leftarrow - \div, \quad
g \leftarrow - g_h.
\end{equation}

To carry out the convergence analysis, we leverage the equivalence between the augmented Lagrangian method and the proximal point algorithm~\cite{Martinet:1970,Rockafellar:1976} for solving a dual problem, as established in~\cite{JPX:2025,Rockafellar:1976b,Setzer:2011}.
By invoking Fenchel--Rockafellar duality~\cite{Rockafellar:1970}, the dual problem associated with~\eqref{general} is given by
\begin{equation}
\label{general_dual}
\min_{q \in W} \left\{ F^*(-B^* q) + (g, q) \right\},
\end{equation}
where $B^* \colon W \rightarrow V$ denotes the adjoint of $B$, and $F^* \colon V \rightarrow \overline{\mathbb{R}}$ is the Legendre--Fenchel conjugate of $F$, defined by
\begin{equation}
\label{Legendre_Fenchel}
F^*(w) = \sup_{v \in V} \left\{ (w, v) - F(v) \right\}, \quad w \in V.
\end{equation}

The proximal point algorithm for solving~\eqref{general_dual} is summarized in \cref{Alg:proximal_point}.

\begin{algorithm}
\caption{Proximal point algorithm for solving~\eqref{general_dual}}
\begin{algorithmic}[]
\label{Alg:proximal_point}
\STATE Given $\epsilon > 0$:
\STATE Choose $q^{(0)} \in W$.
\FOR{$n=0,1,2,\dots$}
    \STATE $\displaystyle
    q^{(n+1)} = \operatornamewithlimits{\arg\min}_{q \in W} \left\{ F^* (- B^* q) + (g, q) + \frac{\epsilon}{2} \| q - q^{(n)} \|^2 \right\}
    $
\ENDFOR
\end{algorithmic}
\end{algorithm}

In \cref{Lem:proximal_point_equiv}, we summarize the equivalence between \cref{Alg:aug_general,Alg:proximal_point}; see~\cite[Remark~6.2]{JPX:2025} for details. 

\begin{lemma}
\label{Lem:proximal_point_equiv}
The augmented Lagrangian method presented in \cref{Alg:aug_general} is equivalent to the proximal point algorithm presented in \cref{Alg:proximal_point} in the following sense: if $p^{(0)} = q^{(0)}$, then we have $p^{(n)} = q^{(n)}$ for all $n \geq 0$.
\end{lemma}

Thanks to \cref{Lem:proximal_point_equiv}, it suffices to analyze \cref{Alg:proximal_point} in order to establish the convergence rate of \cref{Alg:aug_general}.
In \cref{Lem:proximal_point_conv}, we present the convergence result for the proximal point algorithm, as established in~\cite[Example~23.40]{BC:2011}; see also~\cite{Kim:2021,Rockafellar:1976}.

\begin{lemma}
\label{Lem:proximal_point_conv}
In~\eqref{general_dual}, suppose that $F^* \circ (- B^*)$ is $\mu$-strongly convex for some $\mu > 0$, ensuring that~\eqref{general_dual} admits a unique solution $p \in W$.
Then, in the proximal point algorithm presented in \cref{Alg:proximal_point}, we have
\begin{equation*}
\| q^{(n+1)} - p \| \leq \frac{\epsilon}{\mu + \epsilon} \| q^{(n)} - p \|,
\quad n \geq 0.
\end{equation*}
\end{lemma}

\begin{remark}
\label{Rem:local_strong_convexity}
For certain convex optimization problems of the form~\eqref{general_dual} arising in partial differential equations~(see, e.g.,~\cite{LP:2025a}), particularly the dual problem associated with~\eqref{mixed_FEM_opt}, the composition $F^* \circ (- B^*)$ is only \emph{locally} strongly convex.
Nevertheless, since \cref{Alg:proximal_point} is a contraction, the sequence $\{ q^{(n)} \}$ remains bounded.
This boundedness allows the assumption of global strong convexity in \cref{Lem:proximal_point_conv} to be relaxed to a local one.
For simplicity and clarity of the analysis, however, we continue to assume global strong convexity throughout.
A related discussion appears in~\cite[Remark~2.1]{TX:2002}.
\end{remark}

Using \cref{Lem:proximal_point_equiv,Lem:proximal_point_conv}, we establish the convergence rate of \cref{Alg:aug_general}, as stated in \cref{Thm:aug_general}.
We remark that \cref{Thm:aug_general} generalizes~\cite[Lemma~2.1]{LWXZ:2007}, which addresses the special case where $F$ is quadratic, i.e., when~\eqref{general} reduces to a linear saddle-point problem.

\begin{theorem}
\label{Thm:aug_general}
In~\eqref{general_dual}, suppose that $F^* \circ (- B^*)$ is $\mu$-strongly convex.
Then, in the augmented Lagrangian method presented in \cref{Alg:aug_general}, we have
\begin{equation}
\label{Thm1:aug_general}
\| p^{(n + 1)} - p \| \leq \frac{\epsilon}{\mu + \epsilon} \| p^{(n)} - p \|,
\quad n \geq 0.
\end{equation}
Moreover, if we further assume that $F$ is differentiable, then we have
\begin{equation}
\label{Thm2:aug_general}
D_F^{\mathrm{sym}} (u^{(n+1)}, u) \leq
\begin{cases}
\displaystyle
\frac{\mu \epsilon^2}{( \mu + \epsilon)^2} \| p^{(n)} - p \|^2, & \text{ if } \epsilon \leq \mu, \\
\displaystyle
\frac{\epsilon}{4} \| p^{(n)} - p \|^2, & \text{ if } \epsilon > \mu,
\end{cases}
\quad n \geq 0.
\end{equation}
\end{theorem}
\begin{proof}
The first estimate~\eqref{Thm1:aug_general} follows immediately from \cref{Lem:proximal_point_equiv,Lem:proximal_point_conv}.
To prove~\eqref{Thm2:aug_general}, we first see that the optimality condition for~\eqref{general} reads as
\begin{subequations}
\begin{align}
\label{Thm3:aug_general}
\nabla F(u) + B^* p = 0, \\
\label{Thm4:aug_general}
Bu = g.
\end{align}
\end{subequations}
Moreover, \cref{Alg:aug_general} reads as
\begin{subequations}
\begin{align}
\label{Thm5:aug_general}
\nabla F(u^{(n+1)}) + B^* p^{(n)} + \epsilon^{-1} B^* (B u^{(n+1)} - g) = 0, \\
\label{Thm6:aug_general}
p^{(n+1)} = p^{(n)} + \epsilon^{-1} (B u^{(n+1)} - g).
\end{align}
\end{subequations}
Combining~\eqref{Thm3:aug_general},~\eqref{Thm4:aug_general}, and~\eqref{Thm5:aug_general} yields
\begin{equation}
\label{Thm7:aug_general}
\nabla F(u^{(n+1)}) - \nabla F(u) + B^* (p^{(n)} - p) + \epsilon^{-1} B^* B (u^{(n+1)} - u) = 0.
\end{equation}
In addition, combining~\eqref{Thm4:aug_general} and~\eqref{Thm6:aug_general} yields
\begin{equation}
\label{Thm8:aug_general}
p^{(n+1)} - p^{(n)} = \epsilon^{-1} B (u^{(n+1)} - u).
\end{equation}
It follows that
\begin{equation}
\label{Thm9:aug_general}
\begin{split}
D_F^{\mathrm{sym}} &(u^{(n+1)}, u)
\stackrel{\eqref{symmetrized_Bregman}}{=} (\nabla F(u^{(n+1)}) - \nabla F(u), u^{(n+1)} - u) \\
&\stackrel{\eqref{Thm7:aug_general}}{=} - (B^* (p^{(n)} - p), u^{(n+1)} - u) - \epsilon^{-1} (B^* B (u^{(n+1)} - u), u^{(n+1)} - u) \\
&\stackrel{\eqref{Thm8:aug_general}}{=} - \epsilon \| p^{(n+1)} - p \|^2 + \epsilon (p^{(n+1)} - p, p^{(n)} - p) \\
&\leq   \epsilon\left( \frac{1}{2 \theta} - 1 \right) \| p^{(n+1)} - p \|^2 + \frac{\epsilon \theta}{2} \| p^{(n)} - p \|^2,
\end{split}
\end{equation}
where $\theta$ is any positive real number.
On one hand, if $\epsilon \leq \mu$, setting $\theta = \frac{\epsilon}{\mu + \epsilon} \leq \frac{1}{2}$ in~\eqref{Thm9:aug_general} yields
\begin{equation*}
D_F^{\mathrm{sym}} (u^{(n+1)}, u)
\leq \frac{\mu - \epsilon}{2} \| p^{(n+1)} - p \|^2 + \frac{\epsilon^2}{2 (\mu + \epsilon)} \| p^{(n)} - p \|^2
\stackrel{\eqref{Thm1:aug_general}}{\leq} \frac{\mu \epsilon^2}{( \mu + \epsilon )^2} \| p^{(n)} - p \|^2.
\end{equation*}
On the other hand, if $\epsilon > \mu$, setting $\theta = \frac{1}{2}$ in~\eqref{Thm9:aug_general} yields
\begin{equation*}
D_F^{\mathrm{sym}} (u^{(n+1)}, u) \leq \frac{\epsilon}{4} \| p^{(n)} - p \|^2.
\end{equation*}
This completes the proof.
\end{proof}

\begin{remark}
\label{Rem:ALM_importance}
While the arbitrarily fast convergence of the augmented Lagrangian method stated in \cref{Thm:aug_general} can be viewed as a direct consequence of classical results~\cite{BC:2011,Rockafellar:1976b,Rockafellar:1976}, it has seldom been considered in practice because a very small parameter $\epsilon$ typically leads to nearly semicoercive subproblems that are difficult to solve numerically.
In a recent work~\cite{LP:2025b}, it was shown that introducing a suitable multilevel structure can mitigate this near-semicoercivity; see \cref{App:Nearly}.
Thanks to this result, \cref{Thm:aug_general} can now be considered in practical settings, as discussed in this paper.
\end{remark}

\section{Abstract theory of subspace correction methods for nearly semicoercive problems}
\label{App:Nearly}
In this appendix, we provide a brief summary of the abstract convergence theory of subspace correction methods for nearly semicoercive convex optimization problems, as developed in~\cite{LP:2025b}.
Although the theory in~\cite{LP:2025b} is formulated in a broader setting, we confine our discussion here to a specific case suitable for our purpose, for the sake of simplicity and clarity.

Let $V$ be a uniformly smooth and uniformly convex Banach space (cf.~\cite{Megginson:1998}) equipped with a norm $\| \cdot \|$ and a continuous seminorm $| \cdot |$.
We consider the following nearly semicoercive model problem:
\begin{equation}
\label{model_abstract}
\min_{v \in V} \left\{ F(v) := F_0(v) + \epsilon F_1(v) \right\},
\end{equation}
where $F_0 \colon V \rightarrow \mathbb{R}$ and $F_1 \colon V \rightarrow \mathbb{R}$ are G\^{a}teaux differentiable and convex functionals, and $\epsilon$ is a small positive parameter, say $\epsilon \in (0, 1/2]$.
We further assume that $F_0$ is semicoercive with respect to $| \cdot |$~\cite[Proposition~2.3]{LP:2025b}, and that $F_1$ is coercive.

Suppose that the space $V$ admits the following space decomposition:
\begin{equation}
\label{space_decomposition_abstract}
V = \sum_{j=1}^J V_j,
\end{equation}
where each $V_j$, $1 \leq j \leq J$, is a closed subspace of $V$.
The strengthened convexity condition~\cite[Equation~(3.4)]{LP:2025b} asserts that there exists a constant $\tau_0 > 0$ such that
\begin{equation}
\label{strengthened_convexity}
(1 - \tau_0) F(v) + \tau_0 \sum_{j=1}^J F(v + w_j) \geq F \left( v + \tau_0 \sum_{j=1}^J w_j \right),
\quad v \in V, \text{ } w_j \in V_j.
\end{equation}


The parallel subspace correction method for solving the problem~\eqref{model_abstract} based on the space decomposition~\eqref{space_decomposition_abstract} is presented in \cref{Alg:PSC}.

\begin{algorithm}
\caption{Parallel subspace correction method for solving~\eqref{model_abstract}}
\begin{algorithmic}[]
\label{Alg:PSC}
\STATE Given $\tau > 0$:
\STATE Choose $u^{(0)} \in V$.
\FOR{$n=0,1,2,\dots$}
    \FOR{$j = 1, 2, \dots, J$}
        \STATE $\displaystyle w_j^{(n+1)} = \operatornamewithlimits{\arg\min}_{w_j \in V_j} F(u^{(n)} + w_j)$
    \ENDFOR
    \STATE $\displaystyle
    u^{(n+1)} = u^{(n)} + \tau \sum_{j=1}^J w_j^{(n+1)}
    $
\ENDFOR
\end{algorithmic}
\end{algorithm}

In the following, we present conditions that should be verified to show convergence of \cref{Alg:PSC}.
We define the Bregman divergences $d_0$ and $d_1$ corresponding to $F_0$ and $F_1$, respectively, as in~\eqref{d1}.
Combining~\cite[Proposition~4.9 and Assumption~5.4]{LP:2025b} yields the following local smoothness condition.

\begin{assumption}[local smoothness]
\label{Ass:smooth}
For any $\| \cdot \|$-bounded and convex subset $K$ of $V$, we have
\begin{equation*}
    \sup_{v, v + w \in K} \frac{d_{0} (w; v)}{\| w \|^2} < \infty
    \quad \text{and} \quad
    \sup_{v, v + w \in K} \frac{d_{1} (w; v)}{\| w \|^2} < \infty.
\end{equation*}
\end{assumption}

The local uniform convexity condition~\cite[Assumption~5.9]{LP:2025b} is presented as follows.

\begin{assumption}[local uniform convexity]
\label{Ass:uniform}
For some $p \geq 2$, we have the following:
\begin{enumerate}[(a)]
\item For any $| \cdot |$-bounded and convex subset $K$ of $V$, we have
\begin{equation*}
    \mu_{0,K} := \inf_{v, v+w \in K} \frac{d_0 (w; v)}{| w |^p} > 0.
\end{equation*}
\item For any $\| \cdot \|$-bounded and convex subset $K$ of $V$, we have
\begin{equation*}
    \mu_{1,K} := \inf_{v, v+w \in K} \frac{(d_0 + d_1) (w; v)}{\| w \|^p} > 0.
\end{equation*}
\end{enumerate} 
\end{assumption}

Under \cref{Ass:smooth,Ass:uniform}, we can ensure convergence of \cref{Alg:PSC}~\cite[Theorem~4.13]{LP:2025b}~(see also~\cite[Theorem~4.8]{Park:2020}).
However, in order to ensure an $\epsilon$-independent convergence rate, we need more conditions.

The following kernel decomposition assumption~\cite[Assumption~5.6]{LP:2025b}, originally introduced in~\cite{LWXZ:2007} for the analysis of linear problems, plays a critical role in the convergence theory.

\begin{assumption}[kernel decomposition]
\label{Ass:kernel}
The kernel $\cN = \ker F_0$ of the semicoercive functional $F_0$ in~\eqref{model_abstract} admits the decomposition
\begin{equation*}
\cN = \sum_{j=1}^J (V_j \cap \cN).
\end{equation*}
\end{assumption}

The next condition, \cref{Ass:triangle}, describes a triangle-inequality-like property stated in~\cite[Assumption~5.7]{LP:2025b}.

\begin{assumption}[triangle inequality-like property]
\label{Ass:triangle}
For any bounded and convex subset $K \subset V$, there exists a constant $C_{K} > 0$ such that
\begin{equation}
\label{C_tri}
d_1(w + w'; v) \leq C_{K} \left( d_1(w; v) + d_1(w'; v) \right),
\quad v \in K, \quad w, w' \in V.
\end{equation}
\end{assumption}

Similar to~\eqref{K_0}, given an initial guess $u^{(0)} \in V$ for \cref{Alg:PSC}, we define the set $K_0$ by
\begin{equation}
\label{K_0_abstract}
K_0 := \{ v \in V : F(v) \leq F(u^{(0)}) \}.
\end{equation}
Note that $K_0$ depends implicitly on $\epsilon$ due to the $\epsilon$-dependence of the energy functional $F$; see~\cite[Remark~5.5]{LP:2025b}.
Under all the assumptions stated above, we obtain the following convergence theorem, which is a direct consequence of~\cite[Theorems~4.13, 5.8, and 5.10]{LP:2025b}.

\begin{theorem}
\label{Thm:nearly}
Suppose that \cref{Ass:smooth,Ass:uniform,Ass:kernel,Ass:triangle} hold.
In \cref{Alg:PSC}, assume that $\tau \in (0, \tau_{0}]$, where $\tau_0$ was given in~\eqref{strengthened_convexity}.
Then we have the following:
\begin{enumerate}[(a)]
\item In the case $p = 2$, we have
\begin{equation*}
F(u^{(n)}) - F(u)
\leq \left( 1 - \frac{\tau}{2} \min \left\{ 1, \frac{C_2}{C_1} \right\} \right)^n ( F(u^{(0)}) - F(u) ),
\quad n \geq 0,
\end{equation*}
where $C_1$ is a positive constant independent of $\epsilon$, except for its implicit dependence on $K_0$ (see~\eqref{K_0_abstract}), and $C_2 = \min \{ \mu_{0, K_0}, \mu_{1, K_0} \}$.

\item In the case $p > 2$, there exists $\zeta^* > 0$ such that, if $F(u^{(0)}) - F(u) > \zeta^*$, then
\begin{equation*}
F(u^{(1)}) - F(u) \leq \left( 1 - \frac{\tau}{2} \right) ( F(u^{(0)}) - F(u) ),
\end{equation*}
and otherwise,
\begin{equation*}
F(u^{(n)}) - F(u) \leq \frac{ C_1^{\frac{p}{p-2}} }{ \tau^{\frac{p}{p-2}} C_2^{\frac{2}{p-2}} n^{\frac{p}{p-2}} },
\quad n \geq 1,
\end{equation*}
where $C_1$ and $C_2$ are as defined in~(a).
\end{enumerate}
\end{theorem}

\bibliographystyle{siamplain}
\bibliography{refs_Forchheimer}

\begin{thebibliography}{10}

\bibitem{ABB:2013}
{\sc S.~Acharyya, A.~Banerjee, and D.~Boley}, {\em Bregman divergences and
  triangle inequality}, in Proceedings of the 2013 SIAM International
  Conference on Data Mining (SDM), Society for Industrial and Applied
  Mathematics (SIAM), Philadelphia, PA, 2013, pp.~476--484.

\bibitem{ABPR:1986}
{\sc E.~L. Allgower, K.~B{\"o}hmer, F.~Potra, and W.~Rheinboldt}, {\em A
  mesh-independence principle for operator equations and their
  discretizations}, SIAM J. Numer. Anal., 23 (1986), pp.~160--169.

\bibitem{AKW:2013}
{\sc A.~C. Aristotelous, O.~Karakashian, and S.~M. Wise}, {\em A mixed
  discontinuous {G}alerkin, convex splitting scheme for a modified
  {C}ahn-{H}illiard equation and an efficient nonlinear multigrid solver},
  Discrete Contin. Dyn. Syst. Ser. B, 18 (2013), pp.~2211--2238.

\bibitem{AFW:2000}
{\sc D.~N. Arnold, R.~S. Falk, and R.~Winther}, {\em Multigrid in ${H} ({\rm
  div})$ and ${H} ({\rm curl})$}, Numer. Math., 85 (2000), pp.~197--217.

\bibitem{AGPR:2019}
{\sc A.~Arrar{\'a}s, F.~J. Gaspar, L.~Portero, and C.~Rodrigo}, {\em Geometric
  multigrid methods for {D}arcy--{F}orchheimer flow in fractured porous media},
  Comput. Math. Appl., 78 (2019), pp.~3139--3151.

\bibitem{BC:2011}
{\sc H.~H. Bauschke and P.~L. Combettes}, {\em Convex Analysis and Monotone
  Operator Theory in {H}ilbert Spaces}, Springer, New York, 2011.

\bibitem{BBF:2013}
{\sc D.~Boffi, F.~Brezzi, and M.~Fortin}, {\em Mixed Finite Element Methods and
  Applications}, Springer, Heidelberg, 2013.

\bibitem{BV:2004}
{\sc S.~Boyd and L.~Vandenberghe}, {\em Convex Optimization}, Cambridge
  University Press, Cambridge, 2004.

\bibitem{BPX:1990}
{\sc J.~H. Bramble, J.~E. Pasciak, and J.~Xu}, {\em Parallel multilevel
  preconditioners}, Math. Comp., 55 (1990), pp.~1--22.

\bibitem{Brandt:1977}
{\sc A.~Brandt}, {\em Multi-level adaptive solutions to boundary-value
  problems}, Math. Comp., 31 (1977), pp.~333--390.

\bibitem{Brandt:1984}
{\sc A.~Brandt}, {\em Multigrid Techniques: 1984 Guide with Applications to
  Fluid Dynamics}, Gesellschaft f\"ur Mathematik und Datenverarbeitung mbH, St.
  Augustin, 1984.

\bibitem{BS:2008}
{\sc S.~C. Brenner and L.~R. Scott}, {\em The Mathematical Theory of Finite
  Element Methods}, Springer, New York, third~ed., 2008.

\bibitem{BF:2024}
{\sc E.~Bueler and P.~E. Farrell}, {\em A full approximation scheme multilevel
  method for nonlinear variational inequalities}, SIAM J. Sci. Comput., 46
  (2024), pp.~A2421--A2444.

\bibitem{CGKMY:1998}
{\sc X.-C. Cai, W.~D. Gropp, D.~E. Keyes, R.~G. Melvin, and D.~P. Young}, {\em
  Parallel {N}ewton--{K}rylov--{S}chwarz algorithms for the transonic full
  potential equation}, SIAM J. Sci. Comput., 19 (1998), pp.~246--265.

\bibitem{CGKT:1994}
{\sc X.-C. Cai, W.~D. Gropp, D.~E. Keyes, and M.~D. Tidriri}, {\em
  {N}ewton--{K}rylov--{S}chwarz methods in {CFD}}, in Numerical Methods for the
  Navier--Stokes Equations: Proceedings of the International Workshop Held at
  Heidelberg, October 25--28, 1993, Springer, 1994, pp.~17--30.

\bibitem{CK:2002}
{\sc X.-C. Cai and D.~E. Keyes}, {\em Nonlinearly preconditioned inexact
  {N}ewton algorithms}, SIAM J. Sci. Comput., 24 (2002), pp.~183--200.

\bibitem{CKM:2002}
{\sc X.-C. Cai, D.~E. Keyes, and L.~Marcinkowski}, {\em Non-linear additive
  {S}chwarz preconditioners and application in computational fluid dynamics},
  Internat. J. Numer. Methods Fluids, 40 (2002), pp.~1463--1470.

\bibitem{CS:1999}
{\sc X.-C. Cai and M.~Sarkis}, {\em A restricted additive {S}chwarz
  preconditioner for general sparse linear systems}, SIAM J. Sci. Comput., 21
  (1999), pp.~792--797.

\bibitem{CGW:2023}
{\sc L.~Chen, R.~Guo, and J.~Wei}, {\em Transformed primal-dual methods with
  variable-preconditioners}, arXiv preprint arXiv:2312.12355,  (2023).

\bibitem{CHW:2020}
{\sc L.~Chen, X.~Hu, and S.~Wise}, {\em Convergence analysis of the fast
  subspace descent method for convex optimization problems}, Math. Comp., 89
  (2020), pp.~2249--2282.

\bibitem{CDQHLQXL:2023}
{\sc Z.~Chen, Y.~Ding, Z.~Qin, Y.~He, B.~Liang, Y.~Qing, Y.~Xing, and B.~Li},
  {\em Effects of high-velocity flow on the temperature field near the
  wellbore: {A} review}, in 2023 International Conference on Energy
  Engineering, Springer, 2023, pp.~887--912.

\bibitem{CXZ:2008}
{\sc D.~Cho, J.~Xu, and L.~Zikatanov}, {\em New estimates for the rate of
  convergence of the method of subspace corrections}, Numer. Math. Theory
  Methods Appl., 1 (2008), pp.~44--56.

\bibitem{Ciarlet:2002}
{\sc P.~G. Ciarlet}, {\em The Finite Element Method for Elliptic Problems},
  Society for Industrial and Applied Mathematics (SIAM), Philadelphia, PA,
  2002.

\bibitem{DGKKM:2016}
{\sc V.~Dolean, M.~J. Gander, W.~Kheriji, F.~Kwok, and R.~Masson}, {\em
  Nonlinear preconditioning: {H}ow to use a nonlinear {S}chwarz method to
  precondition {N}ewton's method}, SIAM J. Sci. Comput., 38 (2016),
  pp.~A3357--A3380.

\bibitem{DSW:1996}
{\sc M.~Dryja, M.~V. Sarkis, and O.~B. Widlund}, {\em Multilevel {S}chwarz
  methods for elliptic problems with discontinuous coefficients in three
  dimensions}, Numer. Math., 72 (1996), pp.~313--348.

\bibitem{EG:2003}
{\sc E.~Efstathiou and M.~J. Gander}, {\em Why restricted additive {S}chwarz
  converges faster than additive {S}chwarz}, BIT Numer. Math., 43 (2003),
  pp.~945--959.

\bibitem{FA:2020}
{\sc F.~A. Fairag and J.~D. Audu}, {\em Two-level {G}alerkin mixed finite
  element method for {D}arcy--{F}orchheimer model in porous media}, SIAM J.
  Numer. Anal., 58 (2020), pp.~234--253.

\bibitem{FS:2001}
{\sc A.~Frommer and D.~B. Szyld}, {\em An algebraic convergence theory for
  restricted additive {S}chwarz methods using weighted max norms}, SIAM J.
  Numer. Anal., 39 (2001), pp.~463--479.

\bibitem{GGSJH:2000}
{\sc A.~Gavaskar, N.~Gupta, B.~Sass, R.~Janosy, and J.~Hicks}, {\em Design
  guidance for application of permeable reactive barriers for groundwater
  remediation}, tech. report, BATTELLE Project Report, 2000.

\bibitem{GW:2008}
{\sc V.~Girault and M.~F. Wheeler}, {\em Numerical discretization of a
  {D}arcy--{F}orchheimer model}, Numer. Math., 110 (2008), pp.~161--198.

\bibitem{GR:2008}
{\sc R.~Goebel and R.~T. Rockafellar}, {\em Local strong convexity and local
  {L}ipschitz continuity of the gradient of convex functions}, J. Convex Anal.,
  15 (2008), p.~263.

\bibitem{HG:1987}
{\sc S.~M. Hassanizadeh and W.~G. Gray}, {\em High velocity flow in porous
  media}, Transp. Porous Media, 2 (1987), pp.~521--531.

\bibitem{Hestenes:1969}
{\sc M.~R. Hestenes}, {\em Multiplier and gradient methods}, J. Optim. Theory
  Appl., 4 (1969), pp.~303--320.

\bibitem{HT:2000}
{\sc R.~Hiptmair and A.~Toselli}, {\em Overlapping and multilevel {S}chwarz
  methods for vector valued elliptic problems in three dimensions}, in Parallel
  Solution of Partial Differential Equations ({M}inneapolis, {MN}, 1997),
  vol.~120 of IMA Vol. Math. Appl., Springer, New York, 2000, pp.~181--208.

\bibitem{HRD:2012}
{\sc T.~Horneber, C.~Rauh, and A.~Delgado}, {\em Fluid dynamic characterisation
  of porous solids in catalytic fixed-bed reactors}, Micropor. Mesomor. Mat.,
  154 (2012), pp.~170--174.

\bibitem{HCR:2018}
{\sc J.~Huang, L.~Chen, and H.~Rui}, {\em Multigrid methods for a mixed finite
  element method of the {D}arcy--{F}orchheimer model}, J. Sci. Comput., 74
  (2018), pp.~396--411.

\bibitem{JPX:2025}
{\sc B.~Jiang, J.~Park, and J.~Xu}, {\em Connections between convex
  optimization algorithms and subspace correction methods}, arXiv preprint
  arXiv:2505.09765,  (2025).

\bibitem{Kim:2021}
{\sc D.~Kim}, {\em Accelerated proximal point method for maximally monotone
  operators}, Math. Program., 190 (2021), pp.~57--87.

\bibitem{Kuniansky:2016}
{\sc E.~L. Kuniansky}, {\em Simulating groundwater flow in karst aquifers with
  distributed parameter models---{C}omparison of porous-equivalent media and
  hybrid flow approaches}, tech. report, US Geological Survey, 2016.

\bibitem{LP:2009}
{\sc C.-O. Lee and E.-H. Park}, {\em A dual iterative substructuring method
  with a penalty term}, Numer. Math., 112 (2009), pp.~89--113.

\bibitem{LP:2017}
{\sc C.-O. Lee and E.-H. Park}, {\em A dual iterative substructuring method
  with a small penalty parameter}, J. Korean Math. Soc., 54 (2017),
  pp.~461--477.

\bibitem{LP:2025a}
{\sc Y.-J. Lee and J.~Park}, {\em On the linear convergence of additive
  {S}chwarz methods for the $p$-{L}aplacian}, IMA J. Numer. Anal., 45 (2025),
  pp.~2655--2684.

\bibitem{LP:2025b}
{\sc Y.-J. Lee and J.~Park}, {\em Parallel subspace correction methods for
  semicoercive and nearly semicoercive convex optimization with applications to
  nonlinear {PDE}s}, Math. Comp.,  (2025),
  \url{https://doi.org/10.1090/mcom/4172}.

\bibitem{LWC:2009}
{\sc Y.-J. Lee, J.~Wu, and J.~Chen}, {\em Robust multigrid method for the
  planar linear elasticity problems}, Numer. Math., 113 (2009), pp.~473--496.

\bibitem{LWXZ:2007}
{\sc Y.-J. Lee, J.~Wu, J.~Xu, and L.~Zikatanov}, {\em Robust subspace
  correction methods for nearly singular systems}, Math. Models Methods Appl.
  Sci., 17 (2007), pp.~1937--1963.

\bibitem{LY:2001}
{\sc W.~Liu and N.~Yan}, {\em Quasi-norm local error estimators for
  $p$-{L}aplacian}, SIAM J. Numer. Anal., 39 (2001), pp.~100--127.

\bibitem{LMS:2009}
{\sc H.~L{\'o}pez, B.~Molina, and J.~J. Salas}, {\em Comparison between
  different numerical discretizations for a {D}arcy--{F}orchheimer model},
  Electron. Trans. Numer. Anal., 34 (2009), pp.~187--203.

\bibitem{Martinet:1970}
{\sc B.~Martinet}, {\em R\'egularisation d'in\'equations variationnelles par
  approximations successives}, Rev. Fran\c caise Informat. Recherche
  Op\'erationnelle, 4 (1970), pp.~154--158.

\bibitem{MW:2021}
{\sc G.~Medici and L.~J. West}, {\em Groundwater flow velocities in karst
  aquifers; importance of spatial observation scale and hydraulic testing for
  contaminant transport prediction}, Environ. Sci. Pollut. Res. Int., 28
  (2021), pp.~43050--43063.

\bibitem{Megginson:1998}
{\sc R.~E. Megginson}, {\em An Introduction to {B}anach Space Theory},
  Springer-Verlag, New York, 1998.

\bibitem{MDG:2015}
{\sc H.~Mustapha, L.~de~Langavant, and M.~A. Giddins}, {\em {D}arcy and
  non-{D}arcy flows in fractured gas reservoirs}, in SPE Reservoir
  Characterisation and Simulation Conference and Exhibition, SPE, 2015,
  p.~D021S019R001.

\bibitem{Nesterov:2013}
{\sc Y.~Nesterov}, {\em Gradient methods for minimizing composite functions},
  Math. Program., 140 (2013), pp.~125--161.

\bibitem{NN:2009}
{\sc F.~Nielsen and R.~Nock}, {\em Sided and symmetrized {B}regman centroids},
  IEEE Trans. Inform. Theory, 55 (2009), pp.~2882--2904.

\bibitem{PR:2012}
{\sc H.~Pan and H.~Rui}, {\em Mixed element method for two-dimensional
  {D}arcy--{F}orchheimer model}, J. Sci. Comput., 52 (2012), pp.~563--587.

\bibitem{Park:2005}
{\sc E.-J. Park}, {\em Mixed finite element methods for generalized
  {F}orchheimer flow in porous media}, Numer. Methods Partial Differential
  Equations, 21 (2005), pp.~213--228.

\bibitem{Park:2020}
{\sc J.~Park}, {\em Additive {S}chwarz methods for convex optimization as
  gradient methods}, SIAM J. Numer. Anal., 58 (2020), pp.~1495--1530.

\bibitem{Park:2022a}
{\sc J.~Park}, {\em Additive {S}chwarz methods for convex optimization with
  backtracking}, Comput. Math. Appl., 113 (2022), pp.~332--344.

\bibitem{Park:2024a}
{\sc J.~Park}, {\em Additive {S}chwarz methods for semilinear elliptic problems
  with convex energy functionals: {C}onvergence rate independent of
  nonlinearity}, SIAM J. Sci. Comput., 46 (2024), pp.~A1373--A1396.

\bibitem{Powell:1969}
{\sc M.~J.~D. Powell}, {\em A method for nonlinear constraints in minimization
  problems}, in Optimization ({S}ympos., {U}niv. {K}eele, {K}eele, 1968),
  Academic Press, London-New York, 1969, pp.~283--298.

\bibitem{Rockafellar:1970}
{\sc R.~T. Rockafellar}, {\em Convex Analysis}, Princeton University Press,
  Princeton, NJ, 1970.

\bibitem{Rockafellar:1976b}
{\sc R.~T. Rockafellar}, {\em Augmented {L}agrangians and applications of the
  proximal point algorithm in convex programming}, Math. Oper. Res., 1 (1976),
  pp.~97--116.

\bibitem{Rockafellar:1976}
{\sc R.~T. Rockafellar}, {\em Monotone operators and the proximal point
  algorithm}, SIAM J. Control Optim., 14 (1976), pp.~877--898.

\bibitem{SLM:2013}
{\sc J.~J. Salas, H.~L{\'o}pez, and B.~Molina}, {\em An analysis of a mixed
  finite element method for a {D}arcy--{F}orchheimer model}, Math. Comput.
  Model., 57 (2013), pp.~2325--2338.

\bibitem{SMH:2023}
{\sc M.~Sedghi-Asl, E.~Morales-Casique, and S.~M. Hassanizadeh}, {\em
  Dispersion in high-porosity porous medium}, J. Porous Media, 26 (2023),
  pp.~1--12.

\bibitem{Setzer:2011}
{\sc S.~Setzer}, {\em Operator splittings, {B}regman methods and frame
  shrinkage in image processing}, Int. J. Comput. Vis., 92 (2011),
  pp.~265--280.

\bibitem{TX:2002}
{\sc X.-C. Tai and J.~Xu}, {\em Global and uniform convergence of subspace
  correction methods for some convex optimization problems}, Math. Comp., 71
  (2002), pp.~105--124.

\bibitem{Wise:2010}
{\sc S.~M. Wise}, {\em Unconditionally stable finite difference, nonlinear
  multigrid simulation of the {C}ahn--{H}illiard--{H}ele--{S}haw system of
  equations}, J. Sci. Comput., 44 (2010), pp.~38--68.

\bibitem{XZ:2002}
{\sc J.~Xu and L.~Zikatanov}, {\em The method of alternating projections and
  the method of subspace corrections in {H}ilbert space}, J. Amer. Math. Soc.,
  15 (2002), pp.~573--597.

\bibitem{XZ:2017}
{\sc J.~Xu and L.~Zikatanov}, {\em Algebraic multigrid methods}, Acta Numer.,
  26 (2017), pp.~591--721.

\bibitem{YLQ:2015}
{\sc Y.~Yao, G.~Li, and P.~Qin}, {\em Seepage features of high-velocity
  non-{D}arcy flow in highly productive reservoirs}, J. Nat. Gas Sci. Eng., 27
  (2015), pp.~1732--1738.

\bibitem{Zhang:1992}
{\sc X.~Zhang}, {\em Multilevel {S}chwarz methods}, Numer. Math., 63 (1992),
  pp.~521--539.

\bibitem{Zhang:1994}
{\sc X.~Zhang}, {\em Multilevel {S}chwarz methods for the biharmonic
  {D}irichlet problem}, SIAM J. Sci. Comput., 15 (1994), pp.~621--644.

\bibitem{ZZ:1988}
{\sc I.~Zi{\'o}{\l}kowska and D.~Zi{\'o}{\l}kowski}, {\em Fluid flow inside
  packed beds}, Chem. Eng. Process., 23 (1988), pp.~137--164.

\end{thebibliography}

\end{document}